\documentclass[12pt]{article}

\usepackage{amsmath}
\usepackage{amssymb}
\usepackage{amsfonts}
\usepackage{amsthm} 
\usepackage{enumerate}

\newtheorem{theorem}{Theorem}[section]
\newtheorem{proposition}[theorem]{Proposition}
\newtheorem{lemma}[theorem]{Lemma}
\newtheorem{definition}[theorem]{Definition}
\newtheorem{corollary}[theorem]{Corollary}

\numberwithin{equation}{section}

\newcommand{\rr}{{\mathbb R}}
\newcommand{\zz}{{\mathbb Z}}
\newcommand{\nn}{{\mathbb N}}
\newcommand{\cc}{{\mathbb C}}

\newcommand{\qq}{{\mathbb Q}}

\newcommand{\one}{{\bf 1}}

\newcommand{\supp}{\hbox{supp\,}}

\newcommand{\R}{\mathbb{R}}

\newcommand{\C}{\mathbb{C}}

\begin{document}

\title{Buffon's needle estimates for rational product Cantor sets}
\author{Matthew Bond, Izabella {\L}aba, Alexander Volberg}
\date{June 20, 2012}
\maketitle





\allowdisplaybreaks{



\begin{abstract}
Let $S_\infty=A_\infty\times B_\infty$ be a self-similar product Cantor set in the complex plane, defined via
$S_\infty=\bigcup_{j=1}^L T_j(S_\infty)$, where $T_j:\C\to\C$ have the form $T_j(z)=\frac1{L}z+z_j$ and $\{z_1,\dots,z_L\}=A+iB$ for some $A,B\subset\rr$ with $|A|,|B|>1$ and $|A||B|=L$.
Let $S_N$ be the $L^{-N}$-neighbourhood of $S_\infty$, or equivalently (up to constants), its
$N$-th Cantor iteration.
We are interested in the asymptotic behaviour as $N\to\infty$ of the {\it Favard length} of $S_N$, defined as the average (with respect to direction) length of its 1-dimensional 
projections. If the sets
$A$ and $B$ are rational and have cardinalities at most 6, then the Favard length of $S_N$ is bounded from above by $CN^{-p/\log\log N}$ for some $p>0$. The same result holds with no restrictions on
the size of $A$ and $B$ under certain implicit conditions concerning the generating functions of these sets. This generalizes the earlier results of Nazarov-Perez-Volberg, {\L}aba-Zhai, and Bond-Volberg.
\end{abstract}

\section{Introduction}

The \textbf{Buffon needle probability}, or \textbf{Favard length}, of a compact set $S\subset\C$  is defined as
\begin{equation}\label{favard}Fav(S):=\frac1{\pi}\int_0^\pi |proj_\theta(S)|d\theta,\end{equation}
where $proj_\theta$ denotes the orthogonal projection onto the line forming the angle $\theta$ with the positive real axis, and $|F|$ denotes the Lebesgue measure of $F$ regarded as a subset of $\R$. Pointwise, one defines $proj_\theta(re^{i\theta'}):=r\cdot\cos(\theta'-\theta)$.

We will be concerned with the Favard length of small neighbourhoods of one-dimensional self-similar Cantor sets in the plane. Specifically, let $T_1,\dots,T_L:\C\to\C$ be  similarity maps of the form $T_j(z)=\frac1{L}z+z_j$, where $z_1,\dots,z_L$ are distinct and not colinear. We then define $S_\infty$ to be the unique compact set such that $S_\infty=\bigcup_{j=1}^L T_j(S_\infty)$. 
It is well known that such a set exists, has Hausdorff dimension at most 1 (equality follows if the \textbf{open set condition} holds, see  \cite{mattila}) and finite $H^1$ measure.

Since the $z_j$ are not colinear, $S_\infty$ is unrectifiable, hence it follows from a theorem of Besicovitch that $|proj_\theta(S_\infty)|=0$ for almost every $\theta$ (see \cite{mattila}). It follows that
\begin{equation}\label{manticore}
\lim_{N\to\infty}Fav(S_N)=Fav(S_\infty)=0,
\end{equation}
where $S_N$ is the $\epsilon=L^{-N}$-neighborhood of $S_\infty$.
The question of interest concerns the rate of decay in (\ref{manticore}). The first general quantitative upper bound 
$Fav(S_N)\leq C\exp(-C\log^*N)$
is due to Peres and Solomyak \cite{PS1}; here,  $\log^*N$ denotes the number of iterations of the $\log$ function needed to have
$\log\dots\log N\lesssim 1$. (See also \cite{Tao} for a weaker result in a more general setting.) More recently, power type bounds
$Fav(S_N)\leq CN^{-p}$ with $p>0$ have been proved by Nazarov, Peres and Volberg \cite{NPV} for the 4-corner Cantor set, Bond and Volberg \cite{BV1}
for $L=3$ (the 1-dimensional Sierpi\'nski triangle), and {\L}aba and Zhai \cite{LZ} for rational product Cantor sets with a ``tiling" condition, namely that
$|proj_{\theta_0}(S_\infty)|>0$ for some direction $\theta_0$. The best known result for general self-similar sets is $Fav(S_N)\lesssim e^{-c\sqrt{\log N}}$, 
due to Bond and Volberg \cite{BV3}.

In the converse direction, 
Mattila \cite{Mattila} proved that $Fav(S_N)\geq CN^{-1}$ for a wide class of sets including the general self-similar case.
Bateman and Volberg \cite{BV} proved that for the 4-corner set, the lower bound is at least $C N^{-1}\log N,$
showing that at least in this case Mattila's lower bound is not optimal.  However, Peres and Solomyak prove in \cite{PS1} that for ``random 4-corner sets"
the expected asymptotics is in fact $CN^{-1}$.

Our first result is a power bound for general self-similar sets with $L=4$.

\begin{theorem}\label{thm L4}
Suppose $L=4$. Then $Fav(S_N)\lesssim N^{-p}$ for some $p>0$.
\end{theorem}

The main result of this paper concerns the rational product set case, where $\{z_j\}_{j=1}^L =A\times B$ for some $A,B\subset\qq$. 
Without loss of generality, we may assume that $A,B\subset\zz$ and $\min(A)=\min(B)=0$. Define $S_N=A_N\times B_N+\{z\in\cc:|z|<L^{-N}\}$, where $A_1:=A$ and $A_{N+1}:=A_N+L^{-N-1}A$, and similarly for $B$; this is slightly inconsistent with the general definition of $S_N$ that we gave above, but equivalent to it up to constants and more convenient to use.

\begin{theorem}\label{main thm}
If $S_N=A_N\times B_N$ and $|A|,|B|\leq 6$, then $Fav(S_N)\lesssim N^{-p/\log\log N}$ for some $p>0$.
\end{theorem}

We will make no attempt to find $p$ explicitly, since it will have many dependencies and is unlikely to be close to optimal in any case.

The proof of Theorem \ref{main thm} is based on a new method of estimating so-called ``Riesz products" of trigonometric polynomials. The arguments of  \cite{NPV}, with the additional modifications of \cite{BV1}, \cite{BV3}, \cite{LZ}, have reduced the problem to proving lower bounds on integrals of the form
\begin{equation}\label{intro-e1}
\int_{L^{-m}}^1 \prod_{j=1}^{n} |\phi_t(L^{j}\xi)|^2 d\xi,
\end{equation}
where $t=\tan(\theta)$ and
\begin{equation}\label{phit}
\phi_t(\xi):=\frac1{L}\sum_{(a,b)\in A\times B}e^{2\pi i(a+tb)\xi}
\end{equation}
We can write (\ref{phit}) as $\phi_t(\xi)=\phi_A(\xi)\cdot\phi_B(t\xi)$, where
\begin{equation}\label{phiab}
\phi_A(\xi)=\frac1{|A|}\sum_{a\in A}e^{2\pi ia\xi}\text{ and }\phi_B(\xi)=\frac1{|B|}\sum_{b\in B}e^{2\pi ib\xi}
\end{equation}
For general (not necessarily product) self-similar sets, we defiine $\phi_\theta(\xi)=\frac1{L}\sum_{j=1}^L e^{2\pi ir_j\cos(\theta_j-\theta)}$ instead, where
$z_j=r_je^{2\pi i\theta_j}$ are the similarity centers.

The strategy of \cite{NPV}, \cite{BV1}, \cite{BV3}, \cite{LZ} is now as follows. Write the integrand in (\ref{intro-e1}) as $|P_1(\xi)|^2\,|P_2(\xi)|^2$, where
$$
P_1(\xi)=\prod_{j=m+1}^n \phi_t(L^j\xi),\ \ 
P_2(\xi)=\prod_{j=1}^m \phi_t(L^j\xi),
$$
and consider first the easier task of estimating the integral
$\int_0^1|P_1(\xi)|^2d\xi$. The argument is based on an idea due to Salem. 
It does not use any specific information about the high-frequency part $P_1$, except that it can be written as a long trigonometric polynomial $P_1(\xi)=\sum_{\alpha\in \mathcal{A}}
e^{2\pi i\alpha\xi}$, where $\mathcal{A}\subset\zz$ and $|\mathcal{A}|=L^{n-m}$.

We have $P_j(\xi)=\overline{P_j(-\xi)}$, so that $\int_0^1|P_1|^2=\frac{1}{2}\int_{-1}^1|P_1|^2$.
Let $h(x)=\one_{[0,1/2]}*\one_{[-1/2,0]}$, then $0\leq h\leq C$, $\supp h\subset [-1,1]$ and $\widehat{h}\geq 0$. Therefore
\begin{equation}\label{intro-e2}
\begin{split}
\int_0^1|P_1|^2&\geq C^{-1}\int_{-1}^1 |P_1|^2h(\xi)d\xi\\
&=C^{-1}L^{-2(n-m)} \sum_{\alpha,\alpha'}\int h(\xi)e^{2\pi i(\alpha-\alpha')\xi}d\xi\\
&=C^{-1}L^{-2(n-m)} \sum_{\alpha,\alpha'}\widehat{h}(\alpha-\alpha')\\
&\geq C^{-1}L^{-2(n-m)} \sum_{\alpha=\alpha'}\widehat{h}(0)\\
&\geq C^{-1} L^{-2(n-m)}|\mathcal{A}|=C^{-1} L^{m-n}.\\
\end{split}
\end{equation}

What we actually need is a similar bound on the much more difficult integral in (\ref{intro-e1}).
This brings up two issues. 
One is that we are now integrating over the smaller interval $[L^{-m},1]$, so that we need to have some control of the size of $I_0=\int_0^{L^{-m}}|P_1|^2$. This was accomplished in \cite{NPV} already (see also \cite{BV1}, \cite{BV3}), and the estimates proved there are essentially optimal. 

This leaves us with the crux of the matter, namely replacing the high frequency part $P_1$ in (\ref{intro-e2}) by the full trigonometric polynomial $P_1P_2$. Conceivably, the low frequency part $P_2(\xi)$ might only be large on the interval $[0,L^{-m}]$ -- the one that we have just had to exclude - and very small for $|\xi|\geq L^{-m}$. There would be nothing unusual or pathological about it; this is how trigonometric polynomials are wont to behave, and it is the reason why the estimate on $I_0$ in the last paragraph could only work for the high frequency part. 

What we need to know, then, is just how much harm the $P_2$ part can do to the estimate in (\ref{intro-e2}), the key consideration being that what is left of the estimate must still dominate $I_0$.

In this regard, all of the previous results in \cite{NPV}, \cite{BV1}, \cite{BV3}, \cite{LZ} rely on estimates on the size of the \textbf{set of small values} (SSV) of $P_2$, which we now define.
Let $\varphi:\rr\to\cc$. Let $m$, $c_1$, and $L$ be understood from context, and let $\psi$ have $c_1$ as a parameter. Let
\begin{equation}\label{ssvpropdef}
SSV_\psi:=\{\xi\in[0,1]:|\prod_{k=1}^m\varphi(L^k\xi)|\lesssim\psi(m)\}
\end{equation}
(Note that $P_2(\xi)=\prod_{k=1}^m\phi(L^k\xi)$.)

\begin{definition}\label{SSV prop}
We say that $\varphi$ has \textbf{the SSV property with SSV function }$\psi$ if there exist $c_1,c_2,c_3>0$ with $c_3\gg c_2$ such that $SSV_\psi$ is contained in $L^{c_2m}$ intervals of size $L^{-c_3m}$. In decreasing order of strength:
\begin{itemize}
\item If $\psi(m)=L^{-c_1m}$, we say that $\varphi$ has \textbf{the SSV property}.
\item If $\psi(m)=L^{-c_1m\log m}$, we say that $\varphi$ has \textbf{the log-SSV property}.
\item If $\psi(m)=L^{-c_1m^2}$, we say that $\varphi$ has \textbf{the square-SSV property}.
\end{itemize}
\end{definition}

In our application, the function $\varphi$ will be either $\phi_t$ or one of its factors as defined below, and we will need the constants $c_i$ to be uniform in $t$. For product sets, this is automatic, since the SSV property for $\phi_t$ will follow from the SSV property for just two functions $\phi_{A}$ and $\phi_{B}$. For general self-similar sets with $L=4$, the uniformity will be clear from the proof. Furthermore, we will be able to make the ratio $c_3/c_2$ as large as we wish, at the expense of making $c_1$ large.

If $\phi_A$ and $\phi_B$ have the SSV property, this implies a good bound on the contribution to (\ref{intro-e2}) coming from integration on SSV, so that the remaining portion of that integral still dominates $I_0$, and the integral in (\ref{intro-e1}) is bounded from below by $L^{-2c_1m}$ times the estimate in (\ref{intro-e2}). This is good enough to yield a power type bound.
 
For the sets considered in \cite{NPV}, \cite{LZ}, \cite{BV1}, the SSV property indeed holds. In this paper we extend the same argument further to cover a wider class of product Cantor sets. Furthermore, the case $L=4$ (Theorem \ref{thm L4}) follows without much trouble from \cite{BV3} together with an additional estimate supplied in Section \ref{SSV4}. However, this argument fails already for general self-similar sets with $L=5$, as well as for product sets with $|A|=|B|=5$ (Section \ref{ssv fail}). In this generality, only the square-SSV property is available, leading to the weaker bound of \cite{BV3}. Any further progress beyond that requires a different approach.

What saves the day for us is, first, that while $P_2$ may well be unacceptably small on somewhat longer intervals than the SSV or log-SSV property would allow, it is still reasonably large on most of $[0,1]$; and second, that this remaining part of the interval contains a large \textbf{structured} subset. We emphasize that the structure of the ``good" set is crucial to the argument. Indeed, 
a quick glance at the calculation in  (\ref{intro-e2}) again should convince the reader that it is not enough to simply bound $|P_2|$ from below on some large but otherwise arbitrary subset of $[0,1]$. We need to be able to perform Salem's calculation on the set in question, and that calls for a function $h$ which is supported on that set and has a non-negative Fourier transform. Most generic large subsets of $[0,1]$ fail on that count. However, if we can find a \textbf{difference set} $\Gamma-\Gamma$ such that $\Gamma$ is sufficiently large and $P_2$ is large on $\Gamma-\Gamma$, we will see that then the calculation in (\ref{intro-e2}) does go through in the modified setting. 

The use of Salem's argument on difference sets is the first major idea of this paper. While it appears to be new in this context, it is inspired heavily by similar calculations involving the so-called ``Bohr sets" in additive combinatorics (see e.g. \cite{Bourg1999}). However, we will need to be much more particular about choosing the structured set $\Gamma$. Were we to use the Bohr set associated with the exponents in the trigonometric polynomial $P_2(\xi)$ as it is usually defined in additive combinatorics, that set would not be large enough for our purposes, and could in fact be contained entirely in the small interval $[0,L^{-m}]$ that must be excluded from our integral anyway.

The precise statement of the property we need is as follows.

\begin{definition}\label{slv structured}
Let $\varphi:\cc\to\cc$, and let $L,K,m$ be understood from context. 
We say that $\varphi$ is \textbf{SLV-structured} if there is a Borel set $\Gamma\subset[0,1]$ (in our application, a 
finite union of intervals) and constants $C_1,C_2$ such that:
\begin{equation}\label{slv-e1}
\Gamma-\Gamma\subset\{\xi:\ |\prod_{k=1}^m\varphi(L^k\xi)|\geq L^{-C_1 m}\},
\end{equation}
\begin{equation}\label{slv-e2}
|\Gamma|\geq C_2KL^{-m}.
\end{equation}
\end{definition}

For convenience, we will refer to $\Gamma$ as the \textbf{SLV set} for $\varphi$. 
(This is a slight abuse of terminology, since the actual ``Set of Large Values" here is the set on the right side of (\ref{slv-e1}) rather than $\Gamma$, but that set will play no part in this paper since we have no way of using it efficiently.)

The proof of Theorem \ref{main thm} will combine SLV arguments with the SSV approach via factorization of the functions involved, which we now describe.

We write $\phi_t(\xi)=\phi_A(\xi)\phi_B(t\xi)$, where $\phi_A(\xi)=\frac{1}{|A|}A(e^{2\pi i\xi})$, $A(x)$ is the generating function of $A$:
\begin{equation}\label{generating}
A(x)=\sum_{a\in A} x^a,\,\,\,
\end{equation}
and similarly for $B$.  We are assuming that $A,B\subset\{0,1,2,\dots\}$, so that $A(x)$ and $B(x)$ are polynomials in $\zz[x]$. For our present purposes, it will suffice to consider $A$ and $B$ separately. 

\begin{definition}\label{A1234}
We have $A(x)=\prod_{i=1}^4 A^{(i)}(x)$, where each $A^{(i)}(x)$ is a product of the irreducible factors of $A(x)$ in $\zz[x]$, defined as follows (by convention, an empty product is identically equal to 1):
\begin{itemize}
\item $A^{(1)}(x)=\prod_{s\in S_A^{(1)}}\Phi_s(x)$, $S_A^{(1)}=\{s\in\nn:\ \Phi_s(x)|A(x), (s,L)\neq 1\}$,
\item $A^{(2)}(x)=\prod_{s\in S_A^{(2)}}\Phi_s(x)$, $S_A^{(2)}=\{s\in\nn:\ \Phi_s(x)|A(x),$ $(s,L)=1\}$,
\item $A^{(3)}(x)$ is the product of those irreducible factors of $A(x)$ that have at least one root of the form $e^{2\pi i\xi_0}$, $\xi_0\in\rr\setminus\qq$,
\item $A^{(4)}(x)$ has no roots on the unit circle.
\end{itemize}
We then define the \textbf{good} and \textbf{bad} factors $A'$ and $A''$ of $A$:
\begin{equation}\label{a'a''}
A'(x):=A^{(1)}(x)A^{(3)}(x)A^{(4)}(x).\ A''(x):=A^{(2)}(x).
\end{equation}

\end{definition}

Here, $\Phi_s(x)$ denotes the $s$-th cyclotomic polynomial (see (\ref{cyclo})). Let also $\phi_A^{(i)}(\xi)=A^{(i)}(e^{2\pi i\xi})$, $\phi_A'(\xi)=A'(e^{2\pi i\xi})$, and $\phi_A''(\xi)=A''(e^{2\pi i\xi})$.

Clearly, the factor $\phi_A^{(4)}$ does not contribute to the small values of $P_2$ and can be safely ignored. The factor $\phi_A^{(1)}$ has the SSV property; this was used in a weaker and somewhat camouflaged form in \cite{NPV}, \cite{LZ}, \cite{BV1}. Furthermore, we will prove in Proposition \ref{liouville} that $\phi_A^{(3)}$ has the log-SSV property. This turns out to be related to diophantine approximation of logarithms of algebraic numbers on the unit circle. Roughly speaking, we will rely on the fact that if $e^{2\pi i\xi_0}$  is a root of $A(x)$ with $\xi_0\in\rr\setminus\qq$, then $\xi_0$ cannot be approximated too well by rational numbers. The precise statement we will invoke is a variant of Baker's Theorem in transcendental number theory. Combining these results, we see that $\phi_A'$ has at least the log-SSV property.

The failure of anything better than a square-SSV property is thus due to the ``bad" factor $\phi_A''$, indeed unavoidable unless $A''\equiv 1$, and it is this factor that the construction of $\Gamma$ will have to accommodate. Very roughly speaking, the set of ``bad zeroes" $\xi$ of $\phi_A$, viewed as a subset of $\rr$, is a lattice with coarser lattices removed from it. If we choose $\Gamma_0$ to be a neighbourhood of one of these coarser lattices, it will have the structure of an approximate additive group and will also avoid the set of bad zeroes of $\phi_A$. We would like for the set $\Gamma$ to be defined as the intersection of rescaled copies of such sets, but as such intersections may not always have the generic size, we instead choose $\Gamma$ to be an intersection of appropriately chosen approximate cosets of the coarse lattice, using pigeonholing to ensure that $|\Gamma|$ is large enough and keeping in mind $\Gamma-\Gamma$ will still be contained in a neighborhood of the coarse lattice.

The challenge is in ensuring that the coarse lattice, therefore $\Gamma_0$, are large enough.
This will involve a detailed study of the cyclotomic divisors of polynomials with $(0,1)$ coefficients. We will take advantage of the existing literature on vanishing sums of roots of unity, especially \cite{LL}, \cite{Mann}, \cite{PR}, \cite{schoen}, for some basic structural results. The particular question that we are interested in does not appear to have been studied and may well be very difficult, but we have been able to make sufficient progress to resolve the case of sets of cardinality at most 6. 

The assumption that $|A|,|B|\leq 6$ is only used at the last stage of the proof, to ensure that the cyclotomic divisors of $A(x)$ and $B(x)$ have the structure we need. If we instead assume such structure directly, no restrictions on the size of sets are needed. We thus have the following theorems.

\begin{theorem}\label{single prime}
Let $A,B$ be as in Theorem \ref{main thm}, but without the assumption that $|A|,|B|\leq 6$. Suppose that each of $A(x)$ and $B(x)$ 
satisfies the assumptions of Proposition \ref{compatible}.
Then 
$Fav(S_N)\lesssim N^{-p/\log\log N}$ for some $p>0$. 
\end{theorem}

In particular, the assumptions of Proposition \ref{compatible} hold for $A(x)$ if there is at most one $s_A\in\nn$ such that $(s_A,L)=1$ and $\Phi_{s_A}(x)|A(x)$, and similarly for $B(x)$. We will prove this in Section \ref{single divisor}.

\begin{theorem}\label{thm-noliouville}
Let $A,B$ be as in Theorem \ref{main thm} or Theorem \ref{single prime}. Assume that all roots of $A(x)$ and $B(x)$ on the unit circle are roots of unity (in the above notation, $A^{(3)}=B^{(3)}=1$). Then
$Fav(S_N)\lesssim N^{-p}$ for some $p>0$.
\end{theorem}

Theorems \ref{single prime} or \ref{thm-noliouville} are not sufficient to cover all rational product sets, as there are many examples of sets $A$ and $B$ for which the assumptions of Proposition \ref{compatible} fail. It may be possible to modify the construction of $\Gamma$ to accommodate such cases, based on a deeper analysis of cyclotomic divisors of polynomials with $(0,1)$ coefficients. We expect this question to be difficult and to require methods that go well beyond those of the present paper.

Another interesting open question is whether $\phi^{(3)}$ must in fact have the SSV property, not just the log-SSV property; if so, then the stronger power estimate in Theorem \ref{thm-noliouville} would hold without the assumption that $A^{(3)}=B^{(3)}=1$. This would likely involve the aforementioned diophantine approximation issues for logarithms of algebraic numbers. The result we invoke in this paper is very general, and it is possible that stronger estimates might hold for the specific types of algebraic numbers arising in this problem.

\bigskip
\noindent\textbf{Acknowledgement.} The first author is an NSF postdoctoral fellow. The first and second authors are supported in part by NSERC Discovery Grant 22R80520. The third author is supported in part by the NSF grant DMS-0758552.

We are very grateful to Michel Waldschmidt and Anne de Roton for pointing us to Theorem 9.1 in \cite{wald}.


\section{Preliminary reductions}

In this section, we reduce the proof of Theorem \ref{main thm} to proving lower bounds on integrals of the form \eqref{intro-e1}. This will be accomplished in Proposition \ref{total est goal}.
We will follow the method of \cite{NPV}, with only minor modifications. 

\subsection{The counting function and the exceptional direction set $E$}\label{count buffons beast}
We first define the ``counting function'' $f_{n,\theta}:=\sum\chi_{proj_\theta(Q)}$, where $Q=z_j+[0,L^{-n}]^2$ for some $z_j\in A_n\times B_n$ and the sum ranges over all $L^n$ such possible $Q$. That is, $f_{n,\theta}$ counts how many squares lie ``above'' or ``below'' $x$ when the ray forming the angle $\theta$ with the real axis is regarded as the positive ``horizontal'' direction.

Let $K$ be a large number depending on $N$, to be fixed shortly.  Large values of $f_{n,\theta}$ result when the squares form very tall ``stacks'' above $x$ on the $\theta$-axis, and stacks of $K$ or more squares generically favor over the long term a generically $K$-to-one (or greater) projection mapping from $S_N$ to its shadow $\text{proj}_\theta(S_N)$, so that we should expect to see $|\text{proj}_\theta (S_N)|\lesssim 1/K$ for $N$ not that much longer than when we first saw stacks $K$ tall.

As our sets $S_N$ are self-similar, it is appropriate to also consider a maximal version of $f$ so that we may freely learn about many subsets and subproducts from estimates of this single quantity:
$$f_{N,\theta}^*:=\sup_{n\leq N}f_{n,\theta}$$
Finally, we define the set $\tilde E_{N,K}$ of ``bad directions'' $\theta$ where the ``typical'' stacking might not occur:
$$A^*=A_{N,K}^*=\{x:\ f_{N,t}^*(x)\geq K\},\ \ \tilde E_{N,K}=\tilde E=\{t:\ |A_{N,K}^*|\leq K^{-3}\}.$$
We will emphasize and deemphasize the parameters $N,K$ as needed. Intuitively, the directions $\theta\in \tilde E_{N,K}$ are those for which $|proj_\theta(S)|$ decays slowly; for flavor, it can be shown that $\bigcup_{K\in\nn}\bigcap_{N\in\nn}\tilde E_{N,K}=\{\theta:|\text{proj}_\theta(S_N)|\not\to 0\text{ as }N\to\infty\}$.

In \cite{NPV}, two combinatorial lemmas reduced the upper bound in Buffon's needle problem to a question about the size of the set $\tilde E$. Both of these, as well as (\ref{beta}) below, are given in a simplified sub-optimal form. With a slight modification of the definiton of $\tilde E$, it is possible to improve some exponents a little, the end result being an improvement in the value of constants we are not attempting to track anyway.

The first lemma is perhaps not difficult to believe once one has understood the heuristic discussion above and the definition of $\tilde E$.

\begin{lemma}\label{combin1} (\cite{BThesis}, Section 5.2.2; also appearing in \cite{NPV}) For $t\in \tilde E_{N,K}$,
$$\max_{n\leq N} \|f_{N,\theta}\|_2^2\leq cK,$$
\end{lemma}

The second lemma says that there is a rather concrete connection between $|A_{N,K}^*|$ and ``future'' values of $|proj_{\theta}(S_{N'})|$, $N'>>N$. 

\begin{lemma}\label{combin2} (\cite{BThesis}, Section 5.2.1; also appearing in \cite{NPV})
For $\theta\notin \tilde E_{N,K}$ and for $N,K\gg 1$, $|proj_\theta(S_{NK^3})|\lesssim\frac1{K}$.
\end{lemma}

Of course we would like $K$ as large as we can manage depending on $N$. This is what works for us:
\begin{itemize}
\item If the ``good'' factor $\phi'$ has the SSV property, let $K:=N^{\epsilon_0}$ for some $\epsilon_0>0$.
\item If $\phi'$ only has the log-SSV property, $K:=N^{\epsilon_0/\log\log N}$ for some $\epsilon_0>0$.
\end{itemize}
The goal is to prove that
\begin{equation}\label{beta}
|\tilde E|\leq C{K^{-1/2}}.
\end{equation}
Then from Lemma \ref{combin2}, one sees that
\begin{equation}\label{unstuffed squirrel}
\begin{split}
\pi\cdot|Fav(S_{NK^3})|&=\int_{\theta\in E}|proj_{\theta}(S_{NK^3})|d\theta+\int_{\theta\notin E}|proj_{\theta}(S_{NK^3})|d\theta
\\
&\lesssim |\tilde E|+(\pi-|\tilde E|)\frac1{K}\lesssim {K^{-1/2}}.
\end{split}
\end{equation}

In the SSV case, the power bound follows immediately.

\begin{proposition}\label{strong result} (Strong result)
Suppose that for $K\gtrsim N^{\epsilon_0}$, one has $|\tilde E_{N,K}|\lesssim {K^{-1/2}}$. Then $|Fav(S_N)|\lesssim N^{-p}$ for some $p>0$.
\end{proposition}

In the log-SSV case, we can apply \eqref{unstuffed squirrel} to get
$$\text{Fav}(S_{N\cdot N^{3\epsilon_0/\log\log N}})\lesssim N^{-\epsilon_0/2\log\log N}$$
The result now clearly follows by the monotonicity of $\text{Fav}(S_N)$ in $N$.

\begin{proposition}\label{weak result} (Weak result)
Suppose that for $K\gtrsim N^{\epsilon_0/\log\log N}$, one has $|\tilde E_{N,K}|\lesssim{K^{-1/2}}$. Then $|Fav(S_N)|\lesssim N^{-p/\log\log N}$ for some $p>0$.
\end{proposition}


\subsection{Reduction to trigonometric polynomial estimates}\label{reduction steps}

It will be helpful for us to change the variable, $t=\tan(\theta)$. This does no harm as we use symmetry to consider only the case $\theta\in [0,\pi/4]$. After rescaling, for each $z\in A_n\times B_n$, we may write $proj_\theta(z)=a+tb$ for some $a\in A_n,b\in B_n$. It is understood that any object depending on $\theta$ has an analogous version depending on $t$, and we freely change notations now.

We write
$$f_{n,t}=\nu_n*L^n\chi_{[0,L^{-n}]},\text{ where}$$
$$\nu_n=*_{k=1}^n\tilde{\nu}_k, \,\,\,\, \tilde{\nu}_k=\frac1{L}\sum_{(a,b)\in A\times B} \delta_{L^{-k}a+tL^{-k}b}.$$
Of course, $\nu$ and $\tilde{\nu}$ depend on $t$, but we will not need to display that dependence explicitly.

The Fourier transform, then, is a decay term times a self-similar product: $$\hat{f}_{n,t}(\xi)=L^n\hat{\chi}_{[0,L^{-n}]}(\xi)\cdot\prod_{k=1}^n\phi_t(L^{-k}\xi),\,\,\,\phi_t(\xi):=\frac1{L}\sum_{(a,b)\in A\times B}e^{2\pi i(a+tb)\xi}$$

We would like to ignore the Fourier decay convolution factor of $f_{n,t}$ (i.e., $L^N\chi_{[0,L^{-N}]}$), since the mathematics of interest lies in the frequency part, $\hat{\nu}_n$. Some pigeonholing accomplishes this. The following integration over $\tilde E$ and the passing from $\tilde E$ to ${E}\subseteq \tilde E$ is an additional pigeonholing done to ensure that $m$ and $n$ do not depend on $t$.

Because of Theorem $\ref{combin1}$, we have for all $ t\in \tilde E$,

$$K\geq  ||f_{N,t}||^2_2\approx ||\widehat{f_{N,t}}||^2_2\geq C\int_1^{L^{N/2}}{|\widehat{\nu_N}(x)|^2dx},$$
leading to
\begin{equation}
\label{intnuhat}
K\geq \frac{C}{|\tilde E|}\int_{\tilde E}\int_1^{L^{N/2}}{|\widehat{\nu_N}(x)|^2dx}dt
\end{equation}

\begin{itemize}
\item If the SSV property holds for $\phi'$, let $m= c_0\log N$ (rounded to an integer), for $c_0>0$ small enough.
\item If only the log-SSV property holds for $\phi'$, let $m=c_0\frac{\log N}{\log\log N}$, for $c_0>0$ small enough.
\end{itemize}

In both cases, we will choose the constants so that $K=L^{c^*m}$ with $c^*\approx c_0\approx \sqrt{\epsilon_0}$. In particular, we may choose $c^*$ as small as we wish, provided that $\epsilon_0$ is sufficiently small.

Split $[1,L^{N/2}]$ into $N/2$ pieces $[L^k,L^{k+1}]$ and take a sample integral of $|\widehat{\nu_N}|^2$ on a small block $[L^{n-m},L^n]\times \tilde E$, with $n\in [N/4,N/2]$ chosen so that
$$
\frac{1}{|\tilde E|}\int_{\tilde E}{\int_{L^{n-m}}^{L^n}{|\widehat{\nu_N}(x)|^2dx\,dt}}\leq CKm/N\,.
$$
This choice is possible by $\eqref{intnuhat}$. Define 
 $$
 {E}:=\lbrace t\in \tilde E: \int_{L^{n-m}}^{L^n}{|\widehat{\nu_N}(x)|^2dx}\leq 2CKm/N\rbrace\,.
 $$
Suppose that (\ref{beta}) fails; then
\begin{equation}\label{beta-fail2}
|{E}|\geq\frac{1}{2K^{1/2}}.
\end{equation}
(The passing from $\tilde E$ to ${E}$ is a bit of a technicality; they are ``typical'' directions of $\tilde E$.) 

Note that $\widehat{\nu_N}(x)=\prod_{k=1}^N{\phi (L^{-k}x)}\approx\prod_{k=1}^n{\phi (L^{-k}x)}$ for $x\in [L^{n-m},L^n]$. This is because for all such $x$, the terms $\phi(L^{-k}x)\to 1$ rapidly as $k\to\infty$.
Hence for all $t\in {E}$,
\begin{equation}\label{stuffed-squirrel}
\int_{L^{n-m}}^{L^n}{\prod_{k=1}^n{|\phi_t(L^{-k}x)|^2}dx}\leq \frac{CKm}{N}
\end{equation}
Recall that  the constant $\epsilon_0>0$ was used to define $K$. Our goal is to contradict this if (\ref{beta-fail2}) holds and if $\epsilon_0$ is sufficiently small, by showing that 
we can find a $t\in{E}$ and a constant $\alpha$ on which $\epsilon_0$ does not depend such that 

\begin{equation}\label{totalest}
\int_{L^{n-m}}^{L^n}{\prod_{k=1}^n{|\phi_t(L^{-k}x)|^2}dx}\geq cKN^{-\alpha\epsilon_0}.
\end{equation}

Indeed, (\ref{stuffed-squirrel}) and (\ref{totalest}) can both hold only for finitely many $N$:
\begin{itemize}
\item If $m=c_0\log N$, $\log N\gtrsim N^{1-\alpha\epsilon_0}$, i.e., $N\leq N^*$ if $\epsilon_0$ is small enough.
\item If $m=c_0\frac{\log N}{\log\log N}$, $\frac{\log N}{\log\log N}\gtrsim N^{1-\alpha\epsilon_0}$, i.e., $N\leq N^*$ if $\epsilon_0$ is small enough.
\end{itemize}

Proposition \ref{total est goal} below summarizes our conclusions in a form convenient for future use.  We rescale (\ref{totalest}) so that the interval of integration becomes $[L^{-m},1]$. We will also need to split up the rescaled integrand in (\ref{totalest}) into low-frequency and high-frequency parts:\begin{equation}\label{p1def}
P_1(x):=P_{1,t}(x)=\prod_{k=m}^{n-1} \phi_t(L^kx);\,\,\, P_2(x):=P_{2,t}(x)=\prod_{k=0}^{m-1} \phi_t(L^kx)
\end{equation}

\begin{proposition}\label{total est goal}
To prove Theorem \ref{main thm}, we need only prove the following:
Let $\epsilon_0>0$ be sufficiently small, and assume that (\ref{beta-fail2}) holds. Then there is a
$t\in{E}$ such that 
\begin{equation}\label{totalest prime}
\int_{L^{-m}}^1 |P_{1,t}(x)P_{2,t}(x)|^2 dx\geq cKL^{-n}N^{-\alpha\epsilon_0}.
\end{equation}
\end{proposition}

Because of our pigeonholing steps, the interval of integration is restricted to $[L^{-m},1]$; however, our techniques for proving lower bounds will involve integration on sets containing $[0,L^{-m}]$, e.g. $[0,1]$ as in (\ref{intro-e2}).
We therefore need to control what happens on $[0,L^{-m}]$.

\begin{lemma}\label{poisson-lemma}
For $t\in E$, we have
\begin{equation}\label{poisson}
\int_0^{L^{-m}}|P_1|^2\leq C_0KL^{-n}.
\end{equation}
\end{lemma}

The above estimate appears in \cite{BV3} and it also comprises Section 5.4 of \cite{BThesis}. The localized estimate \eqref{poisson} follows via a Poisson kernel computation from the estimates that define $E$. Note that $E$ was defined in terms of a supremum over sub-products of $\hat{f}_n$; the ``Poisson localization'' lemma allows us to obtain such estimates for high-frequency sub-products on such smaller intervals.

For a heuristic comparison to the independent case, observe that $|\phi|^2$ has the average value $1/L$, so that $C_0K$ represents the maximum possible gain in the average product beyond the mere product of averages:
$$\int_a^{a+L^{-m}}|\prod_{k=m}^n\phi_t(L^kx)|^2\leq C_0K\cdot L^{-m}\cdot (1/L)^{n-m}=C_0KL^{-n}$$
In fact, we do not expect to do much better, as $|P_1(0)|=1$ is the maximum, and such factors $\phi(L^k\cdot)$ quickly begin to have size $1$ rather than $1/L$ on much smaller intervals near $0$.

We explained in the introduction that our strategy will be to integrate $|P(\xi)|^2$ on a ``good" set $G\subset [0,1]$ where $P_2$ is bounded away from zero. Lemma \ref{poisson-lemma} explains why we must look for sets $G$ such that
\begin{equation}\label{good est}
\int_{G}|P_1|^2\geq C_2KL^{-n}
\end{equation}
with $C_2>C_0$, so that the integral in (\ref{good est}) dominates \eqref{poisson}; in particular, the interval $[0,L^{-m}]$ can then be removed from $G$ and \eqref{good est} still holds up to a change of constant $C_2$. Given that (\ref{poisson}) is essentially optimal, the requirement (\ref{good est}) cannot be relaxed.


\section{The main argument}

Recall that for each of the sets $A$ and $B$, we decomposed $\phi_A$ and $\phi_B$ into ``good" and ``bad" factors $\phi_A=\phi'_A\phi_A''$, $\phi_B=\phi_B'\phi_B''$ (Definition \ref{A1234}). We claim that:

\medskip

1) The factors $\phi'_A$ and $\phi'_B$ have the SSV property, with the $c_3/c_2$ ratio large enough.

\medskip

2) The product $\phi''_A(\xi)\phi''_B(t\xi)$ is SLV-structured, with all constants uniform in $t$.

\medskip

We defer the proof of 1) and 2) to Section 4 (Proposition \ref{ssv factors}) and Sections 5-6, respectively. For now, we will see how 1) and 2) imply Theorem \ref{main thm}.

In the sequel, we will use the notation
$$P_{1,A}=\prod_{j=m+1}^n\phi_A(L^j\xi)\,\,\, ,P_{2,A}=\prod_{j=1}^m\phi_A(L^j\xi),$$
and similarly for $B$. The notation $P_A'$, $P_A''$, $P_{2,A}'$, etc. should be self-explanatory,
e.g. $P_{2,A}'=\prod_{j=1}^m\phi_A'(L^j\xi)$.


\subsection{The SSV estimate}\label{SSV vs non}

The goal of this subsection is to prove Proposition \ref{ssv goal} and Corollary \ref{ssvcor}. Some notations will need to be established to give precise statements, but the reader is free to accept the proof as a black box with no harm to the rest of the paper. 

Recall the definition of $m$ and $\psi(m)$:
\begin{itemize}
\item If $\phi_A',\phi_B'$ both have the SSV property, $m=c_0\log n$ and $\psi(m)=L^{-c_1m}=N^{-\epsilon}$
\item If $\phi_A',\phi_B'$ both have at least the log-SSV property, $m=c_0\log n/\log\log n$ and $\psi(m)=L^{-c_1m\log m}\approx N^{-\epsilon'}$.
\end{itemize}
Let $SSV_A:=\{\xi\in[0,1]:P_{2,A}'(\xi)\leq \psi(m)\}$. The definition of $SSV_B$ is similar up to mild rescaling:  $SSV_B(t):=\{\xi\in[0,1]:\prod_{k=1}^m |P_{2,B}'(t\xi)|\leq\psi(m)\}$. (If we regard $SSV_B$ as a subset of $\rr$ and not just $[0,1]$, then this is just a different scaling of this larger set restricted to the same interval $[0,1]$ afterwards.) Then 
$$SSV(t)=SSV_\psi(t)\subseteq SSV_A\cup SSV_{B}(t),$$
and $|P_{2,t}'|\gtrsim \psi(m)$ (with $c_1$ replaced by $2c_1$) outside of $SSV(t)$.

The (log-)SSV property says that $SSV(t)$ is contained in $L^{c_2m}$ intervals of size $L^{-c_3m}$, where $c_3>c_2$. Let us also assume that $c_2>2$; if not, then by subdividing intervals, we may easily change the triple of SSV constants $(c_1,c_2,c_3)$ to $(c_1,c_2+2,c_3+2)$ by simply subdividing the SSV intervals.

Our goal in this subsection is to prove the following. 

\begin{proposition}\label{ssv goal}
Suppose that $\epsilon_0$ is small enough, and that
\begin{equation}\label{beta-fail}
|E|\geq \frac{1}{2K^{1/2}}.
\end{equation}
Then
$$I:=\frac1{|E|}\int_0^1\int_{SSV(t)\cap[L^{-m},1]}|P_{1,t}(\xi)|^2 d\xi dt\leq \frac{C_0}{4}KL^{-n}$$
\end{proposition}

\begin{corollary}\label{ssvcor}
Under the same assumptions, 
there exists a $t_0\in {E}$ such that 
$$\int_{SSV(t_0)\cap[L^{-m},1]}|P_{1,t_0}(\xi)|^2 d\xi\leq\frac{C_0}{2}KL^{-n}.$$
\end{corollary}

The proof below, based on separation of variables, will follow closely the analogous argument in \cite{NPV} and \cite{LZ}. We remark that an alternative argument given in \cite[Lemma 22 and Proposition 15]{BV3}, based on a more sophisticated variant of the Poisson localization lemma (Lemma \ref{poisson-lemma} in this paper)  and applicable to general (not necessarily product) self-similar sets, does not appear to be quantitatively strong enough  to work in our setting.

We start with a few reductions:

$$I=\frac1{|E|}\int_0^1\int_{SSV(t)\cap[L^{-m},1]}|P_{1,t}(\xi)|^2 d\xi dt$$
$$=\frac1{|E|}\int_0^1\int_{(SSV_A\cup SSV_B(t))\cap[L^{-m},1]}|P_{1,A}(\xi) P_{1,B}(t\xi)|^2 d\xi dt$$
Hence
$$I\leq\frac1{|E|}\int_0^1\int_{SSV_A\cap[L^{-m},1]}|P_{1,A}(\xi) P_{1,B}(t\xi)|^2 d\xi dt$$
$$+ \frac1{|E|}\int_0^1\int_{SSV_B\cap[L^{-m},1]}|P_{1,A}(\xi) P_{1,B}(t\xi)|^2 d\xi dt$$
$$=:I_A + I_B$$
It now suffices to prove that
\begin{equation}\label{ssv goal2}
I_A,I_B\leq\frac{C_0}{8}KL^{-n}.
\end{equation}

So far, we have used averaging over $E$ to replace the single-variable integral on $SSV(t)$ with a double integral in both $\xi$ and $t$. Furthermore, one of the variables ranges over the entire interval $[0,1]$, where good estimates are available. We now wish to take advantage of the product structure of our set and rewrite each integral in (\ref{ssv goal2}) as a product of two integrals with the variables decoupled. 

The integral on $[0,1]$ will be easy to compute. The second one, on the small SSV intervals, will require some work. We will need to further split $P_1$ into frequency ranges, and we set up the notation for this:
$$A_{m_1}^{m_2}(x)=\prod_{k=m_1+1}^{m_2}A(x^{L^k}),$$
and similarly for $B$. (Note that this is not normalized, so that $P_{1,A}(\xi)=|A|^{m-n}A_m^n(e^{2\pi i\xi})$.) The reason for looking at such a splitting is that high-frequency factors $A_\ell^n$, with $\ell>m$ sufficiently large depending on the constants in the SSV estimates, will be adapted to have good estimates on $SSV_A$, and similarly for $B$.

We will only prove (\ref{ssv goal2}) for $I_A$, the case of $I_B$ being almost identical.

\begin{lemma}\label{exp dichot}
One of the two following things must occur:
\begin{itemize}
\item $Fav(S_n)\lesssim e^{-cn}$
\item $\int_{\xi_0}^{\xi_0+L^{-m_1}} |A_{m_1}^{m_2}(e^{2\pi i\xi})|^2\, d\xi\lesssim |A|^{m_2-m_1}L^{-m_1}$
\end{itemize}
\end{lemma}

\medskip

\noindent\textbf{Proof of Lemma \ref{exp dichot}:}
Suppose first that we have the following ``stacking condition'': for some $j\in\nn$, $B_0^j(x)$ has at least one coefficient $\not\in\{0,1\}$. Geometrically, this means that $L^{-j}$-neighbourhood of $B_j$ contains at most $|B|^{j}-1$ distinct intervals, hence there is a row of self-overlapping discs somewhere in $S_j$. This is an obvious failure of the open set condition, and we claim that in this case $S$ has Hausdorff and Minkowski dimension strictly less than 1. Indeed, 
it is easy to see that $|proj_\theta(S_{kj})|\lesssim\frac{(|A|^j|B|^j-1)^k}{|A|^{jk}|B|^{jk}}=\big(\frac{|B|^j-1}{|B|^j}\big)^k=:\gamma^{-{kj/j}}$ for some $0<\gamma<1$. In particular, Theorem \ref{main thm} is true with an exponential bound $Fav(S_n)\lesssim(\gamma^{1/j})^n$.

Now suppose there is no stacking. Then
$$\int_{\xi_0}^{\xi_0+L^{-m_1}} |A_{m_1}^{m_2}(e^{2\pi i\xi})|^2 d\xi=\int_{\xi_0}^{\xi_0+L^{-m_1}}|\sum_{j=1}^{|A|^{m_2-m_1}}e^{2\pi iL^{m_1}\lambda_j\xi}|^2 d\xi,$$
where the $\lambda_j\in \nn$ are distinct. Then
$$|\sum_{j=1}^{|A|^{m_2-m_1}}e^{2\pi iL^{m_1}\lambda_j\xi}|^2
=\sum_{j_1,j_2=1}^{|A|^{m_2-m_1}}e^{2\pi iL^{m_1}(\lambda_{j_1}-\lambda_{j_2})\xi}$$

These have common period $L^{-m_1}$. Clearly only the diagonal terms survive the integration over the given integral, with value $L^{-m_1}$. As there are $|A|^{m_2-m_1}$ such diagonal terms, the second case holds. This finishes the proof.

\hfill$\square$

\medskip

We now return to the proof of Proposition \ref{ssv goal}. Of course, we assume that the sets $A$ and $B$ are non-stacking from now on, as the main theorem is a triviality in the other case.

Let us separate variables now. Let $(\xi,t)\to (\xi,u)$, where $u=\xi t$, $dt=du/\xi$. Then
$$I_A\leq\frac1{|E|}L^{-2(n-m)}\sum_{j=1}^{L^{mc_{2,A}}}\int_{a_j}^{a_j+L^{-mc_{3,A}}}|A_m^n(e^{2\pi i\xi})|^2\frac{d\xi}{\xi}\int_0^1 |B_m^n(e^{2\pi iu})|^2 du$$
First, Lemma \ref{exp dichot} gives us
$$I_A\leq\frac{|B|^{n-m}}{|E|}L^{-2(n-m)}\sum_{j=1}^{L^{mc_{2,A}}}\int_{a_j}^{a_j+L^{-mc_{3,A}}}|A_m^n(e^{2\pi i\xi})|^2\frac{d\xi}{\xi}$$

The $a_j$ are the left endpoints of the SSV intervals in $[L^{-m},1]$ (we may take $a_j=L^{-m}$ instead if $L^{-m}$ is in the interior). In the above sum, let $a_*$ be $a_j$ such that the summand is maximized. Also let $\ell_A=\ell=mc_{3,A}$. Also drop the $A$ to write $c_2,c_3$ instead of $c_{2,A}, c_{3,A}$ with implied understanding of the dependence. Note also that $\frac1{\xi}\leq L^m$ on $[L^{-m},1]$.

Now we have
$$I_A\leq\frac{|B|^{n-m}L^m L^{mc_2}}{|E|}L^{-2(n-m)}\int_{a_*}^{a_*+L^{-\ell}}|A_m^n(e^{2\pi i\xi})|^2 d\xi.$$
We need a trivial estimate followed by another application of Lemma \ref{exp dichot}.
$$I_A\leq\frac{|B|^{n-m}L^m L^{mc_2}}{|E|}L^{-2(n-m)}\int_{a_*}^{a_*+L^{-\ell}}|A_m^\ell(e^{2\pi i\xi})A_\ell^n(e^{2\pi i\xi})|^2 d\xi$$
$$\leq\frac{|B|^{n-m}L^m L^{mc_2}|A|^{2(\ell-m)}}{|E|}L^{-2(n-m)}\int_{a_*}^{a_*+L^{-\ell}}|A_\ell^n(e^{2\pi i\xi})|^2 d\xi$$
$$\leq\frac{|B|^{n-m}L^m L^{mc_2}|A|^{2(\ell-m)}}{|E|}L^{-2(n-m)}|A|^{n-\ell}L^{-\ell}.$$

We would like to have $I_A\leq \frac{C_0}{8}KL^{-n}.$ The previous estimate implies this if and only if
\begin{equation}\label{ssv-goal3}
|E|\geq\frac{8}{C_0K}\frac{L^{c_2m+m}}{|B|^{\ell-m}}.
\end{equation}
Recall that $K=L^{c^*m}$ and $\ell=c_3m$. If $c_3/c_2$ is large enough, the right side of (\ref{ssv-goal3}) is bounded by $\frac{1}{4}L^{-c^*m/2}=\frac{1}{4}K^{-1/2}$ for large $N$. The conclusion now follows from (\ref{beta-fail}).

\hfill$\square$


\subsection{Salem's argument on difference sets}\label{salem}

Choose $t_0\in E$ as in Corollary \ref{ssvcor}. We want to prove that \eqref{totalest prime} holds for this choice of $t_0$. By Proposition \ref{total est goal}, this implies Theorem \ref{main thm}.

\begin{proposition}\label{gamma want}
Suppose that 
\begin{itemize}
\item $\phi'_A$ and $\phi'_B$ have the SSV property with the ratio $c_3/c_2$ large enough,
\item $\phi''(\xi)=\phi''_A(\xi)\phi''_B(t_0\xi)$ is SLV-structured.
\end{itemize}
Then \eqref{totalest prime} holds for $t=t_0$.
\end{proposition}

\noindent\textbf{Proof:}
Fix $t=t_0$, and let $\Gamma$ be the SLV set for $\phi''=\phi''_{t_0}$. This means that
\begin{equation}\label{e-add1}
|P_2''(\xi)|\geq L^{-C_1 m}\hbox{ on }\Gamma-\Gamma,
\end{equation}
\begin{center}and\end{center}
\begin{equation}\label{e-add2}
|\Gamma|\geq C_2KL^{-m}.
\end{equation}

We write $P_1(\xi)=\sum_{\alpha\in\mathcal{A}}
e^{2\pi i\alpha\xi}$ (note that $|\mathcal{A}|=L^{n-m}$). Observe that $|P_i(\xi)|=|P_i(-\xi)|$, $i=1,2$, so that all integrals are symmetric with respect to reflection $\xi\to -\xi$, e.g. 
$$\int_{L^{-m}}^1 |P_1|^2 |P_2|^2 d\xi = \frac{1}{2} \int_{[-1,1]\setminus [-L^{-m}, L^{-m}]}|P_1|^2 |P_2|^2 d\xi.$$

Let $h=|\Gamma|^{-1}\one_\Gamma*\one_{-\Gamma}$, then $0\leq h \leq 1$ and 
$\widehat{h}=|\Gamma|^{-1}|\widehat{\one_\Gamma}|^2\geq 0$. Hence
\begin{align*}
\int_{\Gamma-\Gamma}|P_1(\xi)|^2&\geq \int_{\Gamma-\Gamma}|P_1(\xi)|^2h(\xi)d\xi\\
&\geq C L^{-2(n-m)}\sum_{\alpha,\alpha'}\int_{\Gamma-\Gamma}h(\xi)e^{2\pi i(\alpha-\alpha')\xi}d\xi\\
&\geq C L^{-2(n-m)}\Big(\sum_{\alpha}\int_{\Gamma-\Gamma}h(\xi)d\xi
+\sum_{\alpha\neq \alpha'}\widehat{h}(\alpha-\alpha')\Big)\\
&\geq C L^{-2(n-m)}L^{n-m}|\Gamma|=CL^{m-n}|\Gamma|.
\end{align*}
By (\ref{e-add2}), we have
\begin{equation}\label{e-add3}
\int_{\Gamma-\Gamma}|P_1(\xi)|^2d\xi \geq 2C_0KL^{-n},
\end{equation}
which dominates (\ref{poisson}). It follows that
$$
\int_{(\Gamma-\Gamma)\setminus[-L^{-m},L^{-m}]}|P_1(\xi)|^2d\xi \geq C_0KL^{-n},
$$
hence using also (\ref{e-add1}),
\begin{align*}
\int_{L^{-m}}^1 |P_1(\xi)|^2\,&|P_2'(\xi)|^2|P_2''(\xi)|^2\,d\xi\\
&\gtrsim L^{-2C_1m} \int_{(\Gamma-\Gamma)\setminus[-L^{-m},L^{-m}]}|P_1(\xi)|^2|P_2'(\xi)|^2 d\xi\\
&\gtrsim L^{-2C_1m}\psi(m)^2 \int_{(\Gamma-\Gamma)\setminus\big[[-L^{-m},L^{-m}]\cup SSV(t_0)\big]}|P_1(\xi)|^2 d\xi\\
\end{align*}

Recalling how $t_0$ was chosen in Corollary \ref{ssvcor},
\begin{align*}
\int_{L^{-m}}^1 |P_1(\xi)|^2\,&|P_2'(\xi)|^2|P_2''(\xi)|^2\,d\xi\\
&\gtrsim L^{-2C_1m}\psi(m)^2 KL^{-n} \\
&\gtrsim KL^{-n}N^{-\alpha\epsilon_0}\\
\end{align*}
for some $\alpha>0$. The last inequality is true by the choice of $m$ and by the SSV or log-SSV property of $\phi'$; in either case, $L^{-C_1m}\geq\psi(m)\gtrsim N^{-\varepsilon m/2}$ for some $\varepsilon>0$. 
This proves Proposition \ref{gamma want}.

\hfill$\square$

\medskip
\noindent\textbf{Remark.}
It is clear from the proof of Proposition \ref{gamma want} that the condition (\ref{e-add2}) cannot be relaxed, as it is just strong enough to ensure that the integral in (\ref{e-add3}) dominates (\ref{poisson}). In particular, $|\Gamma|\geq C_2 KL^{-Cm}$ for some possibly large constant $C$ would  \textbf{not} suffice, by way of contrast to \eqref{e-add1} and many other estimates in this paper where multiplicative constants in the exponent make little difference. It is this dependence which leads us to investigate the cyclotomic divisibility problems of Section \ref{small ab}. Further, it seems likely that an improvement on Section \ref{small ab} can extend Theorem \ref{main thm} to hold for many (or perhaps all) finite $A,B\subset\nn$.


\section{The SSV property}\label{SSVsect}

Recall from Definition \ref{SSV prop} that the Set of Small Values (SSV) of a function $\varphi:\rr\to\cc$ was defined as
$$
SSV_\psi:=\{\xi\in[0,1]:|\prod_{k=1}^m\varphi(L^k\xi)|\lesssim\psi(m)\}.
$$
The SSV property with SSV function $\psi$ asserts that $SSV_\psi$ can be covered by $L^{c_2m}$ intervals of size $L^{-c_3m}$, with $c_2<c_3$. When $\psi(m)=L^{-c_1m}$, $\psi(m)=L^{-c_1m\log m}$ or $\psi(m)=L^{-c_1m^2}$, we refer to this as the SSV property, the log-SSV property, or the square-SSV property, respectively.

The self-similar sets considered in \cite{BV1}, \cite{LZ}, and \cite{NPV} all induced functions $\phi_\theta$ having the SSV property, with no ``bad'' factors $\phi''_\theta$ whatsoever, courtesy of their good ``tiling" behaviour. In particular, the property holds for $L=3$ \cite{BV1}, and for the product set case if there exist $r_1,r_2\in\nn$ such that $r_1A+r_2B$ tiles $\zz$ \cite{LZ}. We will see in Section \ref{SSV4} that it holds for general self-similar sets (not necessarily product sets) with $L=4$.

By way of contrast, in \cite{BV3} {the square-SSV property} was used. In fact, nothing better holds in any suitable generality - see Section \ref{ssv fail}. As such, our $\Gamma$ construction avoids a legitimate obstacle, as we asserted earlier. We will now discuss when such SSV properties do and do not hold.

\subsection{The SSV property holds for $L=3,4$}\label{SSV4}

In the case $L=3$, the SSV property can be seen to hold for all angles. The key observation is that in this case, one may write $3\cdot\phi_t(x)=1+e^{ia\xi}+e^{ibt\xi}$ and observe $\phi_t(\xi)=0$ if and only if the three summands are the complete set of third roots of unity. In particular, $\phi_t(3\xi)=1$ whenever $\phi_t(\xi)=0$, leading to complete separation of the small values of $\phi_t(3^k\cdot),\phi_t(3^{k'}\cdot)$ when $k\neq k'$. \cite{BV1} and \cite{BV3} fill in some details, but elementary methods suffice to derive the SSV property using this observation.

We now prove that the SSV property holds for any fixed set of $4$ non-collinear points. There are similarities to the $L=3$ case discussed above, as will be readily appreciated by examining the below together with \cite{BV1}, \cite{BV3}. In the case $L=4$, the only way for $\phi_t(x)=0$ to occur is for the four terms to form two annihilating pairs. We turn to this fact now and prove the SSV property.

When $L=4$, some normalizations are possible. In fact, three out of four of the similarity centers $z_j$ can be mapped to arbitrary points by an affine map, leaving only one truly free parameter $z_4=r_4e^{i\theta_4}$. Without loss of generality, then, $z_1=0,z_2=1,z_3=i$. Note that
\begin{align*}
\phi_\theta(\xi)&=\frac1{4}\sum_{j=1}^4 e^{ir_j\cos(\theta_j-\theta)\xi}\\
&=\frac1{4}\big[1+e^{i\cos(\theta)\xi}+e^{i\cos(\theta)\tan(\theta)\xi}+e^{ir_4\cos(\theta)[\cos(\theta_4)+\tan(\theta)\sin(\theta_4)]\xi}\big]\\
\end{align*}
By a change of variable $\tan(\theta)\to t$, $\cos(\theta)\xi\to\xi$, we can write
$$
\phi_t (\xi):=\frac14 (1+e^{i\xi} + e^{it\xi} + e^{ig(t)\xi})\,,
$$
where $t\in [-1,1]$ and $g(t)=t\sin(\theta_4)+\cos(\theta_4)$; of course other $\theta$ are handled by symmetry.
So for this $\phi_t=\phi$, we consider
$$
\prod_{k=0}^{m} \phi(4^k\xi)\,.
$$

It will be convenient to argue with trigonometric identities. To do so, we rewrite the observation about annihilating pairs in the form of the estimate
$$
|e^{ix_1}+e^{ix_2}+ e^{ix_3} +e^{ix_4}| \gtrsim \min_{j\neq k;j,k=1,2,3,4} |\cos(\frac{x_j-x_k}{2})|,
$$
where the left-hand side nearly vanishes when the minimum on the right-hand is nearly zero and is nearly attained by two disjoint pairs $(j,k).$ Thus we can easily write the following ``pseudofactorization''

\begin{equation}\label{pseudofac}
|e^{ix_1}+e^{ix_2}+e^{ix_3}+e^{ix_4}|\gtrsim\prod_{1\leq j<k\leq 4}\bigr|\cos\bigr(\frac{x_j-x_k}{2}\bigr)\bigr|
\end{equation}

For us, $x_1=0,x_2=\xi,x_3=t\xi,x_4=g(t)\xi$. Letting $\alpha\xi=\frac1{2}(x_j-x_k)$, it is enough to prove this claim then.
\medskip

\noindent\textbf{Claim:}
The function $\psi(\xi)=\cos(\alpha\xi)$ has the SSV property (with $L=4$) for any $\alpha\in\rr$. One may take $c_3/c_2$ arbitrarily large independent of $\alpha$, and the implied constants appearing in the definition of the SSV property can be made uniform in $\alpha$ for fixed range $\alpha\in[0,\hat{\alpha}]$.

\medskip
\noindent\textbf{Proof of claim:}
%
%
Consider the product of trigonometric functions. Repeating the double angle formula, one gets
$$2^m\sin(x)\cdot\prod_{k=0}^{m-1}\cos(2^kx)=\sin(2^mx)$$
Using the substitution $m\to 2m+1$,
$$2\cdot 4^m\sin(x)\cdot\prod_{k=0}^{2m}\cos(2^kx)=\sin(2\cdot 4^mx)$$
Omitting even terms and reindexing,
%
$$
\bigr|\prod_{k=0}^{m}\cos(4^kx)\bigr|\gtrsim 4^{-m}\bigr|\frac{\sin(2\cdot 4^mx)}{\sin(x)}\bigr|
$$
%
%
Now let $x=\alpha\xi$, where $\xi\in[0,1]$. Then
$$\bigr|\prod_{k=0}^{m}\cos(4^k\alpha\xi)\bigr|\gtrsim 4^{-m}\bigr|\frac{\sin(2\cdot 4^m\alpha\xi)}{\sin(\alpha\xi)}\bigr|$$

But the small values of the right hand side are readily understood. Such $\xi\in\rr$ are contained in this set:
$$(-c^m4^{-m}\alpha^{-1},c^m4^{-m}\alpha^{-1}) + \pi\big[2^{-1}4^{-m}\alpha^{-1}\zz\setminus \alpha^{-1}\zz\big]$$
These intervals can be large for small $\alpha$, but in such a case they do not intersect $[0,1]$ anyway if $c$ is small enough. Otherwise, they are exponentially small, and of the appropriate number. Further, $c_3/c_2$ can be made arbitrarily large by making $c$ small and $c_1$ large. Note that for $\alpha\gtrsim 4^{-m}$, the size of the intervals of small values scales with $\alpha^{-1}$ and their number scales with $\alpha$, and this is the only dependence on $\alpha$. This proves the claim.

$\hfill\square$


\subsection{The SSV property can fail for $L=5$}\label{ssv fail}

Consider the following examples:
\begin{itemize}
\item The product set case with $A=B=\{0,3,4,8,9\}$.
\item The self-similar set with $L=5$ and $z_1=0-i/24$, $z_2= 3/24 +i/24$,
$z_3 = 4/24 -i/24$, $z_4= 8/24 +i/24$, $z_5= 9/24-i/24$. (The imaginary coordinates do not matter in this example other than to avoid collinearity.)
\end{itemize}

Let us rescale slightly from the usual convention so that $\phi_A$ is 1-periodic:
$$
\phi_0(\xi) = \frac15 (1+e^{i\frac{\pi}{4}\xi} +e^{i\frac{\pi}{3}\xi}  + e^{i\frac{2\pi}{3}\xi}  +e^{i\frac{3\pi}{4}\xi})\,.
$$
In the first (product set) example, we have $\phi_A=\phi_B=\phi_0$; in the second example, $\phi_0$ is the trigonometric polynomial corresponding to $\theta=0$. Therefore the calculation below will show the failure of the SSV property for both examples, for all angles in the first case and for $\theta=0$ in the second case. In fact, the square-SSV property is sharp here.

Now let us restrict to the case of the second example; the first example is essentially the same, but one must use $L=25$ instead of $5$.

The reason for the SSV failure is that $\phi_0(5^k \xi)$ has a recurring zero at $\xi=1$:
$\phi_0(1)=\phi_0(5^k) =0, k=1,....$. Therefore, for all
$\xi\in [1- 5^{-200\sqrt{m}},1]$ and $k=0,1,...,\sqrt{m}$ we have
$$
|\phi_0(5^k\xi)|=|\phi_0(5^k\xi)-\phi_0(5^k)|\le C\,5^k|\xi-1|\le C\,5^k\,5^{-200\sqrt{m}}\,.
$$
Let $\Phi_0(\xi) = \prod_{k=0}^m \phi_0(5^k\xi)$, then
\begin{align*}
|\Phi_0(\xi)|&\le |\prod_{k=0}^{\sqrt{m}}\phi_0(5^k\xi)||\prod_{k=\sqrt{m}+1}^m...|\le  |\prod_{k=0}^{\sqrt{m}}\phi_0(5^k\xi)|\\
&\leq C^{\sqrt{m}} \,5^{1+2+\dots + \sqrt{m}} \,(5^{-200\sqrt{m}})^{\sqrt{m}} \le 5^{-100\,m}\,.
\end{align*}

Hence the set of small values includes the entire interval $[1- 5^{-200\sqrt{m}},1]$; in particular, it cannot be covered by $5^{c_2m}$ intervals of length at most $5^{-c_3m}$, $0<c_2<c_3$.

In the second (non-product set) example, the existence of one ``bad" direction $\theta=0$ does not automatically make the SSV approach unviable. In fact, by reading \cite{BV3} carefully one can confirm that if  $\phi_\theta$ satisfies the uniform SSV property 
for all directions $\theta$ except for an exceptional set $\Theta_m$ of size $|\Theta_m|\lesssim e^{-c_4m}$, then we can still get $Fav(S_n)\lesssim n^{-p}$ for some $p>0$. 
However, an additional short calculation shows that for the above example,
the SSV property continues to fail on a set of angles $\Theta_m$ of size $\gtrsim 5^{-c\sqrt{m}}$, which again is far too large. Indeed, for $\theta\in[0,5^{-200\sqrt{m}}]$ we have
$$
|\Phi_{\theta}(\xi)|\le |\prod_{k=0}^{\sqrt{m}}\phi_{\theta}(5^k\xi)|\le  \prod_{k=0}^{\sqrt{m}}(|\phi_{\theta}(5^k\xi)-\phi_{0}(5^k\xi)|+ |\phi_{0}(5^k\xi)|)\,.
$$
The second term in each factor is at most $C\,5^k\,5^{-200\sqrt{m}}$. The first term can be estimated by differentiating in $\theta$ and using the mean value theorem:
$$
|\phi_{\theta}(5^k\xi)-\phi_{0}(5^k\xi)| \le C\, 5^k\, |\theta| \le C\,5^k\,5^{-200\sqrt{m}}\,.
$$
Hence each factor is at most $C\,5^k\,5^{-200\sqrt{m}}$, so that
$$
|\Phi_{\theta}(\xi)|\le C^{\sqrt{m}} \,5^{1+2+\dots + \sqrt{m}} \,(5^{-200\sqrt{m}})^{\sqrt{m}} \le 5^{-100\,m}\,.
$$


\subsection{SSV properties for product sets}\label{section-SSVproduct}

We now return to our study of rational product sets. In this case, $\phi(\xi)=\phi_A(\xi)\cdot\phi_B(t\xi)$, where $\phi_A(\xi)=\frac{1}{|A|}A(e^{2\pi i\xi})$, $A(x):=\sum_{a\in A}x^a$ is a polynomial in $\zz[x]$, and similarly for $B$. We will only consider the SSV properties of $\phi_A$, the case of $\phi_B$ being identical up to the mild rescaling in the SSV definition.

Throughout this subsection, we will refer to the factorization of $A(x)$ given in Definition \ref{A1234}. Clearly, if $\varphi_1,\varphi_2$ have the SSV property, then so does $\varphi_1\cdot\varphi_2$. It follows that we may consider each $A^{(i)}$ separately, and moreover we may split up each $A^{(i)}$ into its own factors as needed.

It was proved in \cite{BV3} that \textbf{any}  exponential polynomial $\varphi(\xi)=\frac1{L}\sum_{j=1}^L e^{2\pi i\lambda_j\xi}$ has at least the square-SSV property, even if the $\lambda_j\in\rr$ are completely arbitrary. However, our proof of Theorem \ref{main thm} requires that the ``good" factor $\phi'_A$ have at least the log-SSV property, and we now turn to this.

Our results regarding the SSV status of the functions $\phi^{(i)}_A(\xi)=\frac{1}{|A|}A^{(i)}(e^{2\pi i\xi})$ are as follows. 

\medskip

1) We prove in Proposition \ref{goodcyc} that $\Phi_s(e^{2\pi i\cdot})$ has the SSV property for $(s,L)\neq 1$. (Equivalently, $\varphi(\xi)=e^{2\pi i\xi}-\zeta$ has the SSV property when $\zeta$ is a root of $\Phi_s$ and $(s,L)\neq 1$.) Therefore $\phi_A^{(1)}$, as a product of such factors, has the SSV property.

\medskip

2) The factor $\phi_A^{(2)}$ has the square-SSV property as discussed above, and \textbf{this is sharp} unless $\phi_A^{(2)}\equiv 1$. Indeed, it is easy to see that the calculation in Section \ref{ssv fail} extends to any $\Phi_s(e^{2\pi i\cdot})$ with $(s,L)=1$.

\medskip

3) Proposition \ref{liouville} shows that the log-SSV property holds for factors of the form $\varphi(\xi)=e^{2\pi i\xi}-e^{2\pi i\xi_0}$, where $\xi_0\in[0,1]\setminus\qq$ and $e^{2\pi i\xi_0}$ is algebraic. This is a consequence of a variant of Baker's Theorem in diophantine number theory. It follows that $\phi_A^{(3)}$ has at least the log-SSV property. We do not know whether this can be improved.

\medskip

4) Clearly, $\varphi(\xi)=e^{2\pi i\xi}-z_0$ has the SSV property whenever $|z_0|\neq 1$.

\medskip

Collecting the cases 1), 3) and 4), we get the following.

\begin{proposition}\label{ssv factors}
The ``good" factor $\phi'_A$ has the log-SSV property. Furthermore, if $A^{(3)}(x)\equiv 1$ (i.e. $A(x)$ has no roots $e^{2\pi i\xi_0}$ with $\xi_0\in\rr\setminus\qq$), then $\phi'_A$ has the SSV property. Furthermore, we can arrange for $c_3/c_2$ to be greater than any given $M>0$, at the cost of increasing $c_1$.
\end{proposition}

We now turn to the proofs of 1) and 3). We will focus on the SSV property itself, but it will be clear from the proofs that $c_1,c_2,c_3$ may be chosen as claimed in the proposition.
Recall that the \textbf{cyclotomic polynomials} $\Phi_s(x)$, $s\in\nn$, are defined as
\begin{equation}\label{cyclo}
\Phi_s(x):=\prod_{d:1\leq d\leq s,(d,s)=1}(x-e^{2\pi id/s}).
\end{equation}
Alternatively, $\Phi_s$ are uniquely determined as the irreducible factors of $x^{M}-1$ in $\zz[x]$:
$$
x^M-1=\prod_{d|M}\Phi_d(x).
$$

\begin{proposition}\label{goodcyc}
Let $s$ and $L$ have a common divisor. Then $\Phi_s(e^{2\pi i\cdot})$ has the SSV property; equivalently, $\varphi(\xi)=e^{2\pi i\xi}-e^{2\pi ik/s}$ has the SSV property for all $(k,s)=1$.
\end{proposition}

\noindent\textbf{Proof.}
The main idea is that $\Phi_s(x)$, $\Phi_s(x^{L^k})$ have no common zeroes, since $(e^{2\pi ij/s})^{L^k}$ is never a primitive $s$-th root of unity by the assumption $(s,L)\neq 1$. The lack of repeated zeroes in the self-similar product is favorable to the SSV condition, but we apply telescoping products to see it directly. 

Let $s=ML_1$, where $L_1|L^a$ for some $a\in\nn$ and $(M,L)=1$.
Let also
$$F(x)=\prod_{k\in\kappa} (x-e^{2\pi ik/{L^a}}),$$
where $\kappa\subset [1,L^a-1]$ is chosen so that $e^{2\pi ik/{L^a}}$ runs through all primitive $L_1$-th roots of unity. The important thing to note is that $\Phi_s(x)|F(x^M)$.

Now consider the ``cyclotomic complement'' of $F$,
$$G(x)=\prod_{k=1,k\not\in\kappa}^{L^a-1}(x-e^{2\pi ik/L^a})$$
It follows that $F(x)\cdot G(x)=\frac{x^{L^a}-1}{x-1}$. Then for $b=0,1,...,a-1$,
$$\prod_{j=0}^{m-1}F(x^{L^{aj+b}})G(x^{L^{aj+b}})=\frac{x^{L^{am+b}}-1}{x^{L^b}-1},$$
so
$$\prod_{j=0}^{am-1}F(x^{L^j})G(x^{L^j})=\prod_{b=0}^{a-1}\frac{x^{L^{am+b}}-1}{x^{L^b}-1}$$

By direct exmination of the small values of the above expression, one can see that $FG(e^{2\pi i\cdot})$ has the SSV property. $1/G(e^{2\pi i\cdot})$ is bounded below and thus has the SSV property, so $F(e^{2\pi i\cdot})=FG/G(e^{2\pi i\cdot})$ has the SSV property.

Hence $F(e^{2\pi iM\cdot})$ has the SSV property by a change of variable. $\Phi_s(x)|F(x^M)$, so dealing with the $H$ in $\Phi_s(x)=F(x^M)/H(x)$ in the same way as with $G$, we finish the proof.

\hfill$\square$

\begin{proposition}\label{liouville}
If $\xi_0\in[0,1]\setminus\qq$ and $e^{2\pi i\xi_0}$ is algebraic, then $\varphi(\xi)=e^{2\pi i\xi}-e^{2\pi i\xi_0}$ has the log-SSV property.
\end{proposition}

\noindent\textbf{Proof.} We will rely on the following application of 
Theorem 9.1 of \cite{wald}: if $\xi_0$ is irrational and $z_0=e^{2\pi i\xi_0}$ is algebraic, then for any integers $a,q$ with $q> 0$ we have
\begin{equation}\label{pseudoliouville}
\Big|\xi_0-\frac{a}{q}\Big|\geq \frac{C_0}{q^{\alpha}},
\end{equation}
where $C_0>0,$ $\alpha>1$ are positive constants that may depend on $\xi_0$, but are independent of
$a,q$. 

To deduce (\ref{pseudoliouville}) from \cite[Theorem 9.1]{wald}, we proceed as follows. We apply Case (ii) of the theorem with $m=2$, $\lambda_1=2\pi i$, $\lambda_2 = 2\pi i\xi_0$, $b_1=-a$, $b_2=q$, so that $\alpha_1=1$ and $\alpha_2=z_0$. We set $E=e$, $E^*=D$, and choose the positive constants $A_1,A_2$ large enough so that the first inequality in the assumptions of the theorem is satisfied. Note that $D,A_1,A_2$ depend on $\xi_0$, but not on $a$ or $q$. 

We may assume that $|a|\leq 2q$, since otherwise (\ref{pseudoliouville}) holds for the trivial reason that $|\xi_0|\leq 1$ and $|a/q|\geq 2$. 
Let $B= 3C_1q$ for some constant $C_1>0$, then $B\geq C_1(|a|+q)$, so that the last assumption of the theorem holds if $C_1$ was chosen large enough. The theorem then states that the quantity
$$\Lambda=b_1\lambda_1+b_2\lambda_2=2\pi iq\Big(\xi_0-\frac{a}{q}\Big)$$
is bounded from below by $\exp(-C_2\log B)\geq \exp(-C_3\log q)=q^{-C_3}$, where, again, the constants may depend on $\xi_0$ but not on $a$ or $q$. This clearly implies (\ref{pseudoliouville}).

We now proceed with the proof of Proposition \ref{liouville}. Let $\varphi(\xi)=e^{2\pi i\xi}-e^{2\pi i\xi_0}$ and $\varphi_k(\cdot)=\varphi(L^k\cdot)$. Then the set of zeroes of $\varphi_k$ is $L^{-k}\xi_0+L^{-k}\zz$. 

Suppose that $\zeta_0\in[0,1]$ is such that 
\begin{equation}\label{liouville-e2}
\prod_{k=1}^m\varphi_k(\zeta_0)\leq L^{-Cm\log m},
\end{equation}
and is not in a $L^{-C'm}$-neighbourhood of any root of $\varphi_k$ for any $k=1,\dots,m$.
(Here and below, we
identify $[0,1]$ with the torus $\mathbb{T}$, and all neighbourhoods and distances are understood accordingly.)

For each $k=1,\dots,m$, let $\xi_k$ be the root of $\varphi_k$ which is nearest to $\zeta_0$, and let $d_k$ be an integer such that $L^{-d_k}\leq |\zeta_0-\xi_k|< L^{-d_k+1}$. Note that $d_k\leq C'm$. Let also $r_k=\max(d_k-k,0)$. Then
$$
\varphi_k(\zeta_0)\geq cL^{-r_k},
$$
with the constant $c$ uniform in $k$, so that 
$$
\prod_{k=1}^m\varphi_k(\zeta_0)\geq c^m L^{-\sum_{k}r_k}.
$$
Assuming that $C$ is large enough, we get from this and (\ref{liouville-e2}) that
\begin{equation}\label{liouville-e3}
\sum_{k=1}^m r_k \geq Cm(\log m)/2.
\end{equation}

We now use (\ref{pseudoliouville}) to get an upper bound on $\sum_{k=1}^m r_k$. 
Let $1\leq k<\ell\leq m$, then for $\xi_k,\xi_\ell$ as above we have
\begin{align*}
|\xi_k-\xi_\ell|&=
\Big|\frac{\xi_0}{L^{k}}+\frac{a_k}{L^{k}} -\frac{\xi_0}{L^{\ell}}-\frac{a_\ell}{L^{\ell}} \Big|
\\
&=L^{-\ell}|L^{\ell-k}\xi_0+L^{\ell-k}a_k-\xi_0-a_\ell|
\\
&=L^{-\ell}|(L^{\ell-k}-1)\xi_0-a|
\\
&=\frac{L^{\ell-k}-1}{L^{\ell}}\,\,\Big|\xi_0-\frac{a}{L^{\ell-k}-1}\Big|
\\
&\geq\frac{L^{\ell-k}-1}{L^{\ell}}\,\,\frac{C_0}{(L^{\ell-k}-1)^\alpha},
\end{align*}
where we substituted $a=a_\ell-L^{\ell-k}a_k$ and used (\ref{pseudoliouville}). Hence
\begin{equation}\label{liouville-e1}
|\xi_k-\xi_\ell|\geq 2C_0L^{-k}L^{-(\ell-k)\alpha}.
\end{equation}
However, we know that $\xi_k$ and $\xi_\ell$ both lie in an interval centered at $\zeta_0$ of length at most $2\max(L^{-d_k},L^{-d_\ell})=2L^{-\min(d_k,d_\ell)}$. It follows that
$$\min(d_k,d_\ell)\leq k+\alpha(\ell-k)+C''\leq \ell+\alpha(\ell-k)+C'',$$
for some constant $C''$. 

We now drop the assumption that $k<\ell$. Interchanging the two indices if necessary, we get that for all $k\neq \ell$,
$$\min(d_k,d_\ell)\leq \min(k,\ell)+\alpha |k-\ell|+C'',$$
so that
\begin{equation}\label{liouville-e4}
\min(r_k,r_\ell)\leq\alpha |k-\ell| +C''.
\end{equation}

For $j=1,2,\dots,\lfloor\log(C'm)/\log 2\rfloor+1$, let 
$$I_k=\{k:\ 1\leq k\leq m,\ 2^{j-1}\leq r_k<2^{j}\}.$$
Fix $j$, and let $k,\ell\in I_j$, $k\neq \ell$. Then by (\ref{liouville-e4}),
$$
|k-\ell|\geq \frac{\min(r_k,r_\ell)-C''}{\alpha}\geq \frac{2^{j-1}-C''}{\alpha}.
$$
We will call $j$ ``large" if $2^{j-1}\geq 2C''$. Then
$$|k-\ell|\geq 2^{j-2}/\alpha,$$
hence $I_j$ can have at most $m/(2^{j-2}\alpha^{-1})=2^{-j+2}\alpha m$ distinct elements.
The number of remaining ``small" values of $j$ is at most $(\log C'')/(\log 2)+2$. We estimate the size of these $I_j$ trivially by $m$, and we note that for $k\in I_j$ with $j$ small we have $r_k\leq 4C''$. We get that
\begin{align*}
\sum_{k=1}^m r_k&=\sum_{j=1}^{\lfloor\log(C'm)/\log 2\rfloor+1}\sum_{k\in I_j}r_k\\
&=\sum_{j\ { small}}\sum_{k\in I_j}r_k
+\sum_{j{\  large}}\sum_{k\in I_j}r_k\\
&\leq \Big(\frac{\log C''}{\log 2}+2\Big)4C''m+\sum_{j{\  large}} 2^{j}2^{-j+2}\alpha m\\
&\leq C^*m\log m,
\end{align*}
where $C^*$ depends on $\alpha$ and the previous constants. This contradicts (\ref{liouville-e3}) if $C$ was chosen large enough.

\hfill$\square$


\section{The construction of $\Gamma$}\label{section-gamma}

The proof of Theorem \ref{main thm} will be complete if we show that for sets $A,B$ as in the theorem, the function $\phi_t''(\xi)$ is SLV-structured for each $t$. 
We will do this by finding sets $\Delta_A,\Delta_B\subset\rr$ and $\Gamma=\Gamma(t)\subset[0,1]$ such that
\begin{equation}\label{e-add1a}
|P_{2,A}''|\geq L^{-C_1m}\text{ on }\Delta_A,
\end{equation}
\begin{equation}\label{e-add1b}
|P_{2,B}''(x)|\geq L^{-C_1m}\text{ on }\Delta_B,
\end{equation}
\begin{center}and\end{center}
\begin{equation}\label{e-add2a}
\Gamma\subset\Delta_A\cap t^{-1}\Delta_B,
\end{equation}
\begin{equation}\label{e-add2b}
|\Gamma|\geq C_2KL^{-m}.
\end{equation}
The inequalities (\ref{e-add1a}) -- (\ref{e-add2a}) imply (\ref{slv-e1}), since then
$$
|P_{2,A}''(x)P_{2,B}''(tx)|\geq L^{-2C_1m}\text{ on }\Delta_A\cap t^{-1}\Delta_B.
$$
Hence $\phi_t''(\xi)$ is SLV-structured with the SLV set $\Gamma$. Recall also from Section \ref{reduction steps} that $K=L^{c^*m}$ for some constant $c^*\approx \sqrt{\epsilon_0}$, which we may choose to be arbitrarily small by letting $\epsilon_0$ be small enough. Thus (\ref{e-add2b}) will follow if we can prove that 
\begin{equation}\label{e-add3a}
|\Gamma|\geq C_2L^{-(1-\epsilon)m}\hbox{ for some }\epsilon>0.
\end{equation}

\subsection{Example with repeated zeroes}\label{repzero}

As a motivating example, we will first construct $\Gamma$ for the self-similar set with $L=5$ appearing in Section \ref{ssv fail} and for the direction $\theta=0$. Then $\phi_0(\xi)=\frac{1}{5}G(e^{2\pi i\xi})$, where $G(x)=1+x^3+x^4+x^8+x^9$. The failure of the SSV property for $\phi_0$ is due to the fact that $\Phi_{12}(x)|G(x)$ and 12 is relatively prime to $5$. The interested reader may check that $G(x)$ has no other roots on the unit circle.

We first construct a set $\Delta_0$ disjoint from the set of small values of $\phi_0$. 
Let $\Lambda=\{\frac{1}{12}, \frac{5}{12}, \frac{7}{12}, \frac{11}{12}\}+\zz$, so that $e^{2\pi i\lambda}$ for $\lambda\in\Lambda$ are exactly the zeroes of $\Phi_{12}$. We want $\Delta_0$ to avoid a neighbourhood of $\Lambda$. The key observation is that all points of $\frac{1}{6}\zz$ are at distance at least $1/12$ from $\Lambda$, hence we may take $\Delta_0$ to be a neighbourhood of $\frac{1}{6}\zz$. We are using here that 6 divides 12, but $\phi_0$ does not vanish at any 6-th root of unity; this is the property that we will try to generalize in the next subsection.

We now turn to the details. Let 
$$\Delta_0=\frac1{6}\zz+\Big(-\frac{\eta}{12},\frac{\eta}{12}\Big)$$
for some $\eta\in(0,1)$. Then there is a constant $c=c(\eta)>0$ such that 
$$\phi_0(\xi)\geq c\hbox{ for }\xi\in\Delta_0.$$
By scaling, we also have
$$\phi_0(5^j\xi)\geq c\hbox{ for }\xi\in\Delta_j:= \frac{5^{-j}}{6}\zz+\Big(-\frac{5^{-j}\eta}{12},\frac{5^{-j}\eta}{12}\Big).
$$
Let $\Delta=\bigcap_{j=0}^{m-1} \Delta_j$, then 
$$\prod_{j=0}^{m-1}|\phi_0(L^j\xi)|^2\geq c^{2m}=L^{-C_1m}\hbox{ for }\xi\in\Delta$$
with $C_1=\frac{\log L}{2\log(1/c)}$. 

It remains to prove that we can choose an $\eta\in(0,1)$ and a set $\Gamma\subset[0,1]$ of size at least $C_25^{-(1-\epsilon)m}$ so that $\Gamma-\Gamma\subset\Delta$. We fix $\eta=1/2$, and let $\tau=(\tau_0,\dots,\tau_{m-1})$ range over all sequences with $\tau_j\in\{0,1,2,3\}$. Define
$$
\Gamma_{\tau,j}=\frac{5^{-j}}{6}\Big(\frac{\tau_j}{4}+\zz\Big)+\Big(0,\frac{5^{-j}}{24}\Big),\ j=0,1,\dots,m-1,
$$
$$
\Gamma_\tau:=[0,1]\cap \bigcap_{j=0}^{m-1}\Gamma_{\tau,j}
$$
Then $\Gamma_{\tau,j}-\Gamma_{\tau,j}\subset\Delta_j$, so that $\Gamma_\tau-\Gamma_\tau\subset\Delta$. Moreover, we have
$$
\bigcup_{\tau\in \{0,1,2,3\}^m}\Gamma_\tau=[0,1]$$
except for the zero measure set of interval endpoints. Hence there is at least one $\tau$ such that $|\Gamma_\tau|\geq 4^{-m}$, which is greater than
$5^{-(1-\epsilon)m}$ for $0<\epsilon<1-\frac{\log 4}{\log 5}$.


\subsection{A more general case}\label{gen gamma}

Let $A,B$ be as in Theorem \ref{main thm}, and fix $m\in\nn$. Our goal is to construct $\Delta_A,\Delta_B\subset\rr$ and $\Gamma\subset[0,1]$ so that (\ref{e-add1a})--(\ref{e-add2a}) and (\ref{e-add3a}) hold.

We begin with $\Delta_A$. Let 
\begin{equation}\label{cond-e0}
s_A:=\text{lcm}\{s:\Phi_s|A''\}=\text{lcm}\{s:\Phi_s|A\text{ and }(s,L)=1\}.
\end{equation}
Then all zeroes of $\phi_A''$ lie in the set $\frac1{s_A}\zz$. We will assume that $s_A>1$,
since otherwise $\phi_A''\equiv 1$ and there is nothing to prove. 
Write $s_A=s_{1,A}s_{2,A}$ with $s_{1,A},s_{2,A}\geq 1$ so that
\begin{equation}\label{cond-1}
\Phi_q(x)\text{ does not divide }A(x)\text{ for any }q|s_{1,A}
\end{equation}

It will be to our advantage to make $s_{1,A}$ as large as possible. It is easy to see that we may choose $s_{1,A}>1$ whenever $s_A>1$. Indeed, let $p$ be any prime divisor of $s_A$, then
$\Phi_p$ cannot divide $A$, since $\Phi_p(1)=p$ and $A(1)=|A|$ are relatively prime by (\ref{cond-e0}). In particular, (\ref{cond-1}) holds with $s_{1,A}=p$. However, this simple choice of $s_{1,A}$ will not always be sufficient.

Fix a choice of $s_{1,A},s_{2,A}$ as in (\ref{cond-1}). Then $\phi_A''(\xi)\neq 0$ for $\xi\in\frac1{s_{1,A}}\zz$, so that
\begin{equation}\label{cond-1a}
\{\xi:\phi_A''(\xi)=0\}\subseteq \frac1{s_{1,A}s_{2,A}}\zz\setminus\frac1{s_{1,A}}\zz
\end{equation}
The key point in (\ref{cond-1a}) is that $\frac1{s_{1,A}}\zz$ is separated by distance at least $\frac{1}{s_{1,A}s_{2,A}}$ from the zeroes of $\phi_A''$. We therefore define $\Delta_{0,A}$ as follows:
$$
\Delta_{0,A}:=\frac1{s_{1,A}}\zz+\Big(\frac{-\eta}{s_{1,A}s_{2,A}},\frac{\eta}{s_{1,A}s_{2,A}}\Big)
$$
for some $\eta\in(0,1)$ to be chosen later. Then 
$$|\phi_A''(\xi)|>c_A\text{ for all }\xi\in\Delta_{0,A}$$
for some constant $c_A=c_A(\eta)>0$. We will not need to worry about the size of $c_A$, as this will only affect the harmless constants in (\ref{e-add1a}). However, it will be important for us to be able to take $\eta$ close to 1.

Let $\Delta_A:=\bigcap_{j=0}^{m-1}\Delta_{j,A}$, where
$$
\Delta_{j,A}:=\frac{L^{-j}}{s_{1,A}}\zz+\Big(\frac{-L^{-j}\eta}{s_{1,A}s_{2,A}},\frac{L^{-j}\eta}{s_{1,A}s_{2,A}}\Big)
$$
Then by scaling,
$$\prod_{j=0}^{m-1}|\phi_A''(L^j\xi)|^2\geq c_A^{2m} \hbox{ for }\xi\in\Delta_A.$$

The set $\Delta_B$ is constructed similarly.
It remains to find a sufficiently large set $\Gamma$. 
This can be done under an additional assumption on $A$ and $B$, which we now state.

\begin{proposition}\label{compatible}
Fix $t\in(0,1)$.
Suppose that we can write $s_A=s_{1,A}s_{2,A}$ with $s_{1,A},s_{2,A}>1$ so that: 
\begin{itemize}
\item  $s_{2,A}<|A|$,
\item $\Phi_q(x)$ does not divide $A(x)$ for any $q|s_{1,A}$,
\end{itemize}
and similarly for $B$. 
Then there is a set $\Gamma\subset [0,1]$ obeying \eqref{e-add2a} and \eqref{e-add3a}.
Consequently, $\phi_t''$ is SLV-structured with the SLV set $\Gamma$.
\end{proposition}

\begin{proof}
The proof uses a pigeonholing argument somewhat similar to that in the last subsection, but with continuous translation parameters. Let
$$\Gamma_{j,A}:=\frac{L^{-j}}{s_{1,A}}\zz+\Big(0,\frac{L^{-j}\eta}{s_{1,A}s_{2,A}}\Big),\ j=0,1,\dots,m-1,$$
and similarly for $B$. Let also $M>0$ be a large number. For $x\in [0,1]$, consider the average
$$
\Psi(x)=\frac{1}{M^{2m}}\int_0^M\dots \int_0^M \prod_{j=0}^{m-1}
\one_{\Gamma_{j,A}}(x+\tau_{j,A})
\prod_{j=0}^{m-1} \one_{t^{-j}\Gamma_{j,B}}(x+\tau_{j,B})$$
$$d\tau_{0,A}\dots d\tau_{m-1,A}d\tau_{0,B}\dots d\tau_{m-1,B}.
$$
This clearly factors as a product of integrals. We have
$$\int_0^M \one_{\Gamma_{j,A}}(x+\tau_{j,A})d\tau_{j,A}
=\Big| [0,M]\cap (\Gamma_{j,A}-x)\Big|
\geq (M-1)\frac{\eta}{s_{2,A}},$$
since the interval $[0,M]$ contains $\lfloor M\rfloor \geq M-1$ disjoint intervals of length $1$, each of which intersects $\Gamma_{j,A}$ in a set of measure exactly $\frac{\eta}{s_{2,A}}$. By a similar argument,
$$\int_0^M \one_{t^{-1}\Gamma_{j,B}}(x+\tau_{j,B})d\tau_{j,B}
=\Big| [0,M]\cap (t^{-1}\Gamma_{j,B}-x)\Big|
\geq (M-t^{-1}) \frac{\eta}{s_{2,B}}
$$
It follows that for all $x\in[0,1]$,
$$
\Psi(x)\geq \Big(\frac{(M-1)(M-t^{-1})}{M^2}\frac{\eta^2}{s_{2,A}s_{2,B}}\Big)^m.
$$
In particular, the same lower bound holds for 
$\int_0^1\Psi(x)dx$. 
But on the other hand, rearranging the integrals we get that
$$
\int_0^1\Psi(x)dx
=\frac{1}{M^{2m}}\int_0^M\dots \int_0^M \Big|[0,1]\cap \bigcap_{j=0}^{m-1}(\Gamma_{j,A}-\tau_{j,A})
\cap \bigcap_{j=0}^{m-1}(t^{-1}\Gamma_{j,B}-\tau_{j,B})\Big|$$
$$d\tau_{0,A}\dots d\tau_{m-1,A}d\tau_{0,B}\dots d\tau_{m-1,B}.
$$
Hence we may choose the translation parameters $\tau_{0,A}$,...,$\tau_{m-1,B}$ so that the set
in the integrand has measure at least
$$
\Big(\frac{(M-1)(M-t^{-1})}{M^2}\frac{\eta^2}{s_{2,A}s_{2,B}}\Big)^m.
$$
Call this set $\Gamma$. If $s_{2,A}< |A|$ and $s_{2,B}< |B|$, then $s_{2,A}s_{2,B}<L$,
so that (\ref{e-add3a}) holds if $M$ was chosen large enough (depending on $t$) and $\eta$ is
sufficiently close to 1. Finally, (\ref{e-add2a}) holds because 
$$(\Gamma_{j,A}+\tau)- (\Gamma_{j,A}+\tau)\subset\Delta_{j,A}$$
for all $j=0,1,\dots,m-1$ and $\tau\in\rr$, and similarly for $B$.


\end{proof}


\section{The cyclotomic divisors of $A(x)$}\label{small ab}

We asserted that our assumptions were valid when $|A|,|B|\leq 6$, or when $A(x),B(x)$ have each at most one ``bad" cyclotomic divisor. To demonstrate this, a change of perspective is in order. To study whether $\Phi_s$ divides $A$ for different $s$-values, we will evaluate $A$ at a particular $s$-th root of unity $\zeta_s$ and examine the set of summands of $A(\zeta_s)$ directly for varying $s$ values, taking advantage of projective relationships of the form $(\zeta_{mn})^m=\zeta_{n}$.

\subsection{Sums of roots of unity}

Let $\zeta$ be a primitive $s$-th root of unity.
We can profitably study whether $\Phi_s|A$ by examining the following unordered tuple, which may have repetitions:
\begin{equation}\label{As}
\mathcal{A}_s=\{\zeta^a\}_{a\in A}.
\end{equation}
Then $\Phi_s|A$ if and only if
\begin{equation}\label{phis div a}
\sum_{a\in A} \zeta^a=\sum_{\eta\in\mathcal{A}_s}\eta=0,
\end{equation}
(Note that $\zeta$ and $\mathcal{A}_s$ is defined modulo the Galois group, but whether (\ref{phis div a}) holds does not depend on the choice of the representative.)

More generally, we will consider equations of the form
\begin{equation}\label{lamprelation}
\sum_{j=1}^J z_j\zeta_j=0,
\end{equation}
where $\zeta_j$ are roots of unity. 
Such equations have been studied e.g. in \cite{deB}, \cite{LL}, \cite{Mann}, \cite{PR}, \cite{Re1}, \cite{Re2}, \cite{schoen}.

We will use $\mathcal{U}$ to denote the set of all roots of unity. We also define $\zz\mathcal{U}$ 
of all finite unordered tuples (with repetitions allowed) $\{z_j\zeta_j\}_{j=1}^J$, $z_j\in\zz$, $\zeta_j\in\mathcal{U}$. (The numbers $z_j$ are sometimes called \textbf{weights}.)
This is a module over $\zz$ when equipped with the following equivalence relation and operations of summation and multiplication by integers.
\begin{itemize}
\item We regard $\mathcal{T}_1=\{z_j\zeta_j\}_{j=1}^J,\mathcal{T}_2=\{w_j\eta_j\}_{j=1}^{J'}$ as \textbf{equivalent} if
$$\sum_{j:\zeta_j=\zeta}z_j=\sum_{j:\eta_{j}=\zeta}w_j$$
for all $\zeta\in\mathcal{U}$
\item Summation is concatenation modulo the equivalence relation:
$$\mathcal{T}_1+\mathcal{T}_2=\{a_j\alpha_j\}_{j=1}^{J+J'},\text{ where}$$
$$a_j=z_j,\alpha_j=\zeta_j\text{ for }j=1,...,J, a_{j+J}=\eta_j,\alpha_{j+J}=\eta_j\text{ for }j=1,...,J'$$
equivalently, collect like $\zeta\in\mathcal{U}$ and sum coefficients; i.e.,
$$\mathcal{T}_1+\mathcal{T}_2=\{(\sum_{j:\zeta_j=\zeta}{z_j}+\sum_{j:\eta_j=\zeta}w_j)\zeta\}_{\zeta\in\{\alpha_k\}_{k=1,...,J+J'}}.$$
\item For $z\in\zz$, $z\mathcal{T}_1:=\{(zz_j)\zeta_j\}\in\zz\mathcal{U}$.
\end{itemize}

In the above, the elements of $\zz\mathcal{U}$ may be viewed as formal linear combinations of roots of unity, 
where $\mathcal{U}$ is regarded simply as a set with no structure. Of course, $\zz\mathcal{U}$ may also be equipped with a natural structure of a group ring, but we will not use this here, except that we need to define \textbf{rotations} of elements of $\zz\mathcal{U}$ by roots of unity. 
For $\eta\in\mathcal{U}$, $\eta\mathcal{T}_1:=\{z_j(\eta\zeta_j)\}_{j=1}^J$, where $\eta\zeta_j$ is multiplied in $\mathcal{U}$. 

 We use set notation for elements of $\zz\mathcal{U}$, reserving summations for genuine sums of complex numbers. 
To distinguish between these more clearly, we define the evaluation map $\sigma$. For any $\mathcal{T}=\{z_j\zeta_j\}_{j=1}^J\in\zz\mathcal{U}$,
\begin{equation}\label{sigmap}
\sigma(\mathcal{T}):=\sum_{j=1}^J z_j\zeta_j
\end{equation}
For example, $\{-1\cdot 1\}$ is different in $\zz\mathcal{U}$ from $\{1\cdot e^{\pi i}\}$, but $\sigma(\{-1\cdot 1\})=\sigma(\{1\cdot e^{\pi i}\})=-1$. We will often omit the coefficient $1$ where it occurs, e.g. 
$\{e^{\pi i}\}=\{1\cdot e^{\pi i}\}$.

%

In this language, the kernel of $\sigma$ is exactly the set of those  $\{z_j\zeta_j\}\in\zz\mathcal{U}$ for which \eqref{lamprelation} holds. Lam and Leung \cite{LL} make use of this observation by casting such problems in the language of tensor products over group rings. We do not use this machinery here, though.

For $n\geq 2$, let \textbf{the $n$-gon in standard position} refer to the unordered tuple $\mathcal{P}_n:=\{e^{2\pi ij/n}\}_{j=0}^{n-1}$. An $n$-gon (not necessarily in standard position) is any rotation $\eta\mathcal{P}_n$ of $\mathcal{P}_n$, where $\eta$ is a root of unity of arbitrary order. Collectively, they are called \textbf{polygons}. 

A theorem of R\'edei-de Bruijn-Schoenberg \cite{Re1}, \cite{Re2}, \cite{deB}, \cite{schoen} (see also \cite{LL}) states that all relations of the form \eqref{lamprelation} are generated by prime polygons, in the following sense: if $\sigma(\mathcal{T})=0$, then 
\begin{equation}\label{lamprey}
\mathcal{T}=\sum_{j=1}^J z_j\eta_j\mathcal{P}_{p_j},
\end{equation}
where $z_j\in\zz$, $\eta_j$ are rotations and $p_j$ are (not necessarily distinct) primes. 
We will therefore refer to any $\{z_j\zeta_j\}_{j=1}^J\in\zz\mathcal{U}$ satisfying  
\eqref{lamprelation} as a \textbf{Linear Multi-Polygon Relation} (LMPRe, or lamprey, for short). 
For our purposes, the genuinely multi-polygon relations with $J\geq 2$ will be of particular importance.

A lamprey is \textbf{irreducible} if no proper subset of it is a lamprey. For example, any prime polygon $\mathcal{P}_p$ is irreducible. It is tempting to think that all irreducible lampreys have this form; however, this is not true. For example, let
$$\mathcal{L}_{5:3}:=\{e^{2\pi i/5},e^{4\pi i/5},e^{6\pi i/5},e^{8\pi i/5},e^{5\pi i/3},e^{7\pi i/3}\}.$$
(In \cite{PR}, $\mathcal{L}_{5:3}$ is called an irreducible relation between roots of unity of type $R(5:3)$.) 
We have
$$0=\sum_{\zeta\in\mathcal{P}_5}\zeta - \sum_{\zeta'\in\mathcal{P}_3}\zeta'+[e^{2\pi i/3}+e^{4\pi i/3}]\sum_{\zeta''\in\mathcal{P}_2}\zeta''=\sum_{j=1}^4 e^{2\pi ij/5} + e^{5\pi i/3} + e^{7\pi i/3}$$
That is, $\mathcal{L}_{5:3}$ sums to $0$, but none of its proper subsets do. Note that the weights of $\mathcal{L}_{5:3}$ are all positive; one can say $\mathcal{L}_{5:3}\in\nn\mathcal{U}$. In fact, this is the smallest irreducible positive-weighted lamprey which is not a $p$-gon, and the number of possible cases rapidly grows beyond this point, though \cite{PR} classifies all such cases for sets having at most $12$ points. 

We will also need to study power mappings. The reason for this is that for $\zeta$ a primitive $mn$-th root of unity, $A(\zeta^m)=0$ if and only if $\Phi_n|A$, that is, if and only if $\{\zeta^{am}\}_{a\in A}$ is a lamprey. However, $\{\zeta^a\}_{a\in A}$ may or may not be a lamprey in such a case, so that $\Phi_{mn}|A$ and $\Phi_{mn}\nmid A$ are both possible. Conversely, the opposite could easily hold as well.

For any $\mathcal{T}=\{z_j\zeta_j\}_{j=1}^J\in\zz\mathcal{U}$ and for any $m\in\zz^+$, we define
\begin{equation}\label{pmap}
\pi_m(\mathcal{T}):=\{z_j\zeta_j^m\}_{j=1}^J
\end{equation}
%
%
\medskip

\noindent
\textbf{Example:} Recall the definition \eqref{As}. Let $A=\{0,1,2\}$, $s=9$, and $\zeta=e^{2\pi i/9}$. Then $\mathcal{A}_9=\{1,e^{2\pi i/9},e^{4\pi i/9}\}$ is not a lamprey. However, $\mathcal{A}_3=\pi_3(\mathcal{A}_9)=\{1,e^{2\pi i/3},e^{4\pi i/3}\}$ is a lamprey. But then $\mathcal{A}_1=\pi_3(\mathcal{A}_3)=\{1,1,1\}$.

That is, $\mathcal{A}_9$ was not a lamprey. However, taking third powers of the points in $\mathcal{A}_9$ resulted in $\mathcal{A}_3$, a triangle and therefore a lamprey. But taking third powers again, we get $\mathcal{A}_1$, the triple point at $1$.

More generally, for primes $p\neq q$, $\pi_p$ sends any $q$-gon to another $q$-gon; in fact, $\pi_p(\eta\mathcal{P}_q)=\eta^p\mathcal{P}_q$; that is, $\pi_p\textbf{ preserves }q\textbf{-gons}$. Conversely, $\pi_p(\eta\mathcal{P}_p)=p\cdot\eta^p:=\{\eta^p,...,\eta^p\}$; that is, $\pi_p\textbf{ collapses }p\textbf{-gons}$ (to the point $\eta^p$).
%

\subsection{The SSV property for $|A|=2,3,4,6$}

Now we would like to show that for $|A|=2,3,4,6$ and for $|B|$ arbitrary, $\phi_A$ has the (log-)SSV property (with no $\phi_A''$ factor whatsoever).

Suppose first that $|A|=6$.
We will show that the cyclotomic roots of $A$ can only be zeroes of $\Phi_s$ for some $s$ divisible by 2 or 3, so that in particular $(s,L)\neq 1$.

Suppose that $\Phi_s$ divides $A$, and consider the lamprey $\mathcal{A}_s:=\{\zeta^a\}_{a\in A}$.
A 6-point lamprey can only take these forms:

\begin{itemize}
\item $\mathcal{A}_s$ can be a union of triangles
\item $\mathcal{A}_s$ can be a union of three line segments ($2$-gons)
\item $\mathcal{A}_s$ can be a rotation of $\mathcal{L}_{5:3}$
\end{itemize}
(This is e.g. in \cite{PR}. Note that a $6$-gon  belongs to both of the first two cases.)

%

Recall that we are assuming that $0\in A$. Hence in the third case, $\mathcal{A}_s$ consists of $30$-th roots of unity, some of them primitive. It follows that  $s=30\cdot\text{gcd}(A)$, so that $s$ and $L$ are both divisible by 6.

Now look at the first case. Again, one triangle must be in standard position, so that $s$ must be divisible by $3$. Similarly, in the second case $s$ must be divisible by $2$.

The cases $|A|=2,3,4$ are easier and left to the reader.


\subsection{The case $|A|=5$}\label{a is 5}

The case $|A|=5$ is different. For such sets, we have seen in Section \ref{ssv fail} that it is possible for $\Phi_s$ with $(s,L)=1$ to divide $A(x)$. Indeed, the lamprey $\{\zeta^a\}_{a\in A}$ can be a union of a triangle and a line segment, so that $s$ need only be divisible by 6 and can quite easily be relatively prime to $L=|A||B|=5|B|$. (The lamprey is {\it parasitic}: it sucks up most of the integral $\int_{L^{-m}}^1|P_1|^2d\xi$, leaving only a skeletal good set $\Gamma$.) Moreover, there may be many such divisors $\Phi_s$. 

We now show that we can ``sidestep the parasitic lamprey.''

\begin{proposition}\label{eelimination}
Let $A\subset\nn$ with $|A|=5$, and let $S_A=\{r:\ \Phi_r(x)|A(x)$ and $(r,|A|)=1\}$. Then there are $j_0,k_0$, depending only on $A$, such that any $s\in S_A$ has the form $s=2^{j_0}3^{k_0}M_s$ for some $M_s$ with $(M_s,6)=1$.
In particular, if $s_0=lcm(S_A)$, then $\mathcal{A}_q$ is not a lamprey for any $q|\frac{s_0}{2}$ or $q|\frac{s_0}{3}$.

\end{proposition}

It follows that $A$ satisfies the assumptions of Proposition \ref{compatible}, with $s_2$ equal to either 2 or 3.

\medskip

\noindent\textbf{Proof.}
Suppose for a contradiction that there are $s'=2^{j'}3^{k'}M'$ and $s''=2^{j''}3^{k''}M''$ such that $\mathcal{A}_{s'},\mathcal{A}_{s''}$ are parasitic lampreys. Since $\pi_{m}$ preserves line segments and triangles when $(m,6)=1$, it follows that  $\pi_{M'}(\mathcal{A}_{s'})=\mathcal{A}_{2^{j'}3^{k'}}$ is still a parasitic lamprey, and similarly for $s''$. We may therefore assume that $M'=M''=1$.

Now we can reduce cases further. Since $\pi_2$ collapses line segments and $\pi_3$ collapses triangles, it is rather immediate that neither $s'|s''$ nor $s''|s'$. So WLOG, $j'<j''$ and $k'>k''$. Let $j_0=j''-j'$, $k_0=k'-k''$. In particular, $\pi_{2^{j_0}}(\mathcal{A}_{s''})=\pi_{3^{k_0}}(\mathcal{A}_{s'})=\mathcal{A}_{(s',s'')}$.

Consider the lamprey $\mathcal{A}_{s'}$. It contains a triangle $\text{tri}_{s'}$ and a line segment $\text{seg}_{s'}$. In particular, $\pi_{3^{k_0}}$ collapses $\text{tri}_{s'}$ to a triple point. Call it $\{\zeta,\zeta,\zeta\}$. Let $\pi_{3^{k_0}}(\text{seg}_{s'})=\{\eta_1,\eta_2\}$; it is a line segment.

Using analogous naming conventions coming from the other side, the map $\pi_{2^{k_0}}$ collapses $\text{seg}_{s''}$ to a double point, but preserves $\text{tri}_{s''}$. As at most one point of $\text{tri}_{s''}$ maps to $\zeta$, $\pi_{2^{k_0}}(\text{seg}_{s''})=\{\zeta,\zeta\}$. But then the other two points of $\text{tri}_{s''}$ must map to the line segment $\{\eta_1,\eta_2\}$, a contradiction.

\hfill$\square$

\subsection{A single divisor}\label{single divisor}

Finally, suppose that $|A|$ is arbitrary, but $A(x)$ has only one cyclotomic divisor $\Phi_s$ such that $(s,L)=1$. Write the lamprey $\mathcal{A}_s$ as a union of irreducible lampreys. Each of those has a decomposition (\ref{lamprey}), and by \cite[Theorem 1]{Mann}, the primes $p_j$ obey $p_j|s$ and $p_j\leq |A|$. Since we are assuming that $(s,|A|)=1$, we must in fact have $p_j<|A|$. We may therefore apply Proposition \ref{compatible} with $s_2=p_j$ for any $j$.



}

\bibliographystyle{amsplain}

\noindent{\sc Bond, {\L}aba: Department of Mathematics, University of British Columbia, Vancouver,
B.C. V6T 1Z2, Canada}

\smallskip
                                                                                     
\noindent{\it  bondmatt@math.ubc.ca, ilaba@math.ubc.ca}

\medskip

\noindent{\sc Volberg: Department of Mathematics, Michigan State University, East Lansing, MI 48824, U.S.A.}

\smallskip
                                                                                     
\noindent{\it  volberg@math.msu.edu}

\end{document}